\documentclass[reqno,a4paper,final]{amsart}

\usepackage[foot]{amsaddr}
\usepackage{graphicx,enumerate,nicefrac,bm}
\usepackage{cite}
\usepackage{amsmath, amssymb}
\usepackage{a4wide}
\usepackage{algorithm,algpseudocode,tabularx}
\usepackage{xcolor}

\makeatletter
\newcommand{\multiline}[1]{%
  \begin{tabularx}{\dimexpr\linewidth-\ALG@thistlm}[t]{@{}X@{}}
    #1
  \end{tabularx}
}
\makeatother

\newcommand{\patch}{\varpi}
\newcommand{\diff}{\mathtt{dec}}
\newcommand{\TkN}{\T^\kappa_N}
\newcommand{\p}{\vec p}
\newcommand{\unj}{u^{n_j}}
\newcommand{\R}{\mathbb{R}}
\newcommand{\E}{\mathsf{E}}

\newcommand{\G}{\mathsf{G}}

\newcommand{\x}{\vec{x}}
\renewcommand{\S}{\mathcal{S}}
\renewcommand{\d}{\mathsf{d}}
\newcommand{\dx}{\,\mathsf{d}\x}

\renewcommand{\L}[1]{\mathrm{L}^{#1}(\Omega)}

\newcommand{\T}{\mathcal{T}}
\newcommand{\V}{\mathbb{V}}
\newcommand{\Vh}{\widehat{\mathbb{V}}}
\newcommand{\Wh}{\widehat{\mathbb{W}}}
\newcommand{\X}{\mathbb{X}}
\newcommand{\bb}{\vec{A}}
\renewcommand{\H}{\mathbb{H}}
\newcommand{\Hdual}{\H^\star}

\renewcommand{\vec}[1]{\bm{\mathsf{#1}}}
\newcommand{\norm}[1]{\left\|#1\right\|}
\newcommand{\Ltwo}{\mathrm{L}^2(\Omega)}
\newcommand{\Lnorm}[1]{\left\|#1\right\|_{\Ltwo}}
\newcommand{\Lfour}{\mathrm{L}^4(\Omega)}
\newcommand{\Lfnorm}[1]{\left\|#1\right\|_{\Lfour}}
\newcommand{\dprod}[1]{\langle #1\rangle}
\newcommand{\tol}{\mathtt{tol}}
\newcommand{\incNk}{\mathtt{inc}_N^n}
\newcommand{\dENk}{\diff^n_N}
\newcommand{\refmesh}{\T^{\kappa,{\tt ref}}_{N}}

\DeclareMathOperator*{\Div}{div}
\DeclareMathOperator{\Span}{span}
\DeclareMathOperator*{\essinf}{ess\,inf}

\newcommand{\myState}[1]{\State\parbox[t]{\dimexpr\linewidth-\algorithmicindent}{#1\strut}}

\newtheorem{theorem}{Theorem}[section]

\newtheorem{proposition}[theorem]{Proposition} 
\newtheorem{corollary}[theorem]{Corollary}

\theoremstyle{definition}

\newtheorem{remark}[theorem]{Remark}

\title[Gradient Flow adaptive FEM for the GPE with rotation]{Gradient Flow Finite Element Discretisations with Energy-Based $hp$--Adaptivity for the Gross-Pitaevskii Equation with Angular Momentum}

\author[P.~Heid]{Pascal Heid$^{1}$}
\address{1) Kantonsschule Zürcher Unterland, 8180 Bülach, Switzerland}

\author[P.~Houston]{Paul Houston$^2$}
\address{2) School of Mathematical Sciences, University of Nottingham,  University 
Park, Nottingham, NG7 2RD, UK}

\author[B.~Stamm]{Benjamin Stamm$^{3}$}
\address{3) Institute of Applied Analysis and Numerical Simulation, University of Stuttgart, D--70569 Stuttgart, Germany}

\author[T.~P.~Wihler]{Thomas P.~Wihler$^4$}
\address{4) Mathematics Institute, University of Bern, CH--3012 Bern, Switzerland}

\begin{document}

\begin{abstract}
This article deals with the \emph{stationary Gross-Pitaevskii non-linear eigenvalue problem in the presence of a rotating magnetic field} that is used to model macroscopic quantum effects such as Bose-Einstein condensates (BECs). 
In this regime, the ground-state wave-function can exhibit an a priori \emph{unknown number of quantum vortices} at \emph{unknown locations}, which necessitates the exploitation of adaptive numerical strategies. 
To this end, we consider the conforming finite element method in combination with a discrete Sobolev gradient descent, which is guided by the energy-topology of the problem, to address the nonlinearity. In addition, a key novelty of this work is an $hp$-adaptive strategy that is solely based on energy decay rather than a posteriori error estimators for the refinement process. Numerical results demonstrate that the $hp$-adaptive strategy is highly efficient in terms of accuracy to compute the ground-state wave function and energy for several test problems where we observe exponential convergence.
\end{abstract}



\maketitle

\section{Introduction}
In quantum physics, the {Gross--Pitaevskii equation (GPE)} \cite{einstein1924quantentheorie,bose1924plancks,dalfovo1999theory} is widely used to study macroscopic quantum effects such as Bose--Einstein condensates (BECs), superconductivity or superfluidity. 
In its physical interpretation of the state of a collection of bosonic particles at zero Kelvin, or very small temperature, it can be rigorously derived from the many-body Schr\"odinger equation as the wave function collapses to a symmetric tensor product of a single-particle function. Such a Hartree--Fock ansatz becomes exact in the (dilute) mean-field limit, see \cite{lieb2001bosons,erdHos2010derivation,lewin2015mean} for some rigorous results.

The GPE exists both in its time-dependent and steady-state form. 
Within this work, we focus on the latter one with the presence of an external stirring magnetic field with magnetic vector potential~$\bb$ that can trigger the formation of quantum vortices with a quantized vorticity. 
Since the number of such vortices, as well as their location, is \emph{unknown} a priori, the use of adaptive finite element strategies is highly desirable. 
In this paper, we present a novel $hp$--adaptive finite element gradient flow method to minimise the Gross--Pitaevskii energy functional \emph{with the inclusion of a rotating magnetic field}, given by

\begin{align} \label{eq:GPenergy}
 \E(u):=\int_\Omega \left(\frac{1}{2} | \nabla u|^2+V(\x)|u|^2+\frac{\beta}{2} |u|^4- i \omega \overline{u} \left(\bb(\x) \cdot \nabla  u\right)\right) \dx,
\end{align}
on the $\Ltwo$-unit sphere
\begin{equation}\label{eq:SH}
\S_{\H}:=\{v \in \H: \Lnorm{v}=1\}.
\end{equation}
Here, $\H:=\mathrm{H}^1_0(\Omega;\mathbb{C})$ denotes the Sobolev space of all complex-valued $\mathrm{H}^1$-functions on a given bounded, connected, and open Lipschitz domain $\Omega\subset\mathbb{R}^d$, in dimensions $d\in\{2,3\}$, with vanishing trace on the boundary~$\partial \Omega$, equipped with the inner-product 
\begin{align*}
(u,v)_\H:=\int_\Omega \nabla u \cdot \nabla \overline{v} \dx
\end{align*}
and induced norm 
\begin{align} \label{eq:Hnorm}
\norm{u}_{\H}^2:=(u,u)_{\H}=\int_\Omega|\nabla u|^2\dx.
\end{align} 

In the energy functional~\eqref{eq:GPenergy}, the (given) function $V \in \mathrm{L}^\infty(\Omega;\mathbb{R})$ represents an external potential, and $\beta \in \mathbb{R}_{\geq 0}$ and $\omega \in \mathbb{R}_{\geq 0}$ are prescribed non-negative constants that control the magnitude of the non-linearity and the rotation speed of the magnetic field, respectively. 
A thorough introduction to the basic theory and mathematical properties of ground-states of the Gross--Pitaevskii equation in a rotating frame can be found in, e.g.,~\cite{bao2005ground,bao2013mathematical,bao2014mathematical}, which includes a discussion regarding the existence and non-existence of ground-states.

We now shortly review the most relevant literature concerning numerical methods for the approximation of solutions to the GPE equation with rotating magnetic field. 
The first key-question to address is whether to consider the problem as an energy minimisation problem or to solve the related non-linear eigenvalue problem that is obtained from looking at critical points of the associated Lagrangian built upon the energy functional~\eqref{eq:GPenergy} and the constraint~\eqref{eq:SH}. Due to the non-linear constraint, the energy minimisation can be viewed as a Riemannian energy minimisation problem. The relation between the Riemannian energy minimisation and the non-linear eigenvector problem is elaborated in~\cite{henning2025gross}. 
The second question that arises is how to discretise the associated problem, i.e., how to represent the unknown function $u$ in a finite-dimensional subspace of $\H$. 

First, early contributions considered the imaginary time method originating from the physics literature, see, e.g.,~\cite{aftalion2001vortices} where the critical angular velocities that lead to vortex formation are studied.
Later, contributions of Bao et al.~\cite{bao2003ground,bao2003numerical,BaoDu:04} introduced a mathematical framework to tackle the problem directly by energy minimisation, recently further analysed in~\cite{10.1093/imanum/draf100}. 
Different metrics to consider the Sobolev gradient of the energy functional have then been introduced in~\cite{danaila2010new}, while \cite{danaila2017computation,antoine2017efficient} considers conjugate gradient type methods. 
In~\cite{chen2023second}, the authors consider a second-order flow involving the second time-derivative of the energy-functional.  
Very recently, Henning and Yadav contributed substantially to the analysis of the gradient descent method~\cite{henning2025convergence}, the finite-element a priori convergence analysis~\cite{henning2025discrete} and the numerical study of conjugate-type gradient descents~\cite{ai2025riemannian}.
Finally, we note that the problem can also be tackled by Newton-type methods, see, e.g., \cite{Wu:2017,altmann2021j} and 
\cite{xu2021multigrid} limited to the non-rotating case.

Second, the ground-state solution $u$ can be discretised in different ways covered by standard techniques such as the finite difference~\cite{BaoDu:04,danaila2010new} or the finite element method~\cite{aftalion2001vortices,bao2003ground,bao2003numerical,danaila2010new,ai2025riemannian,henning2025convergence,hauck2024positivity}, as well as (pseudo-) spectral methods~\cite{BaoDu:04,antoine2017efficient,danaila2017computation,chen2023second}.
In the spirit of our proposed $hp$-strategy, which aims to design highly accurate approximation spaces with few number of degrees of freedom, are the recent approaches~\cite{peterseim2024super,henning2023optimal,döding2025} based on the Localized Orthogonal Decomposition (LOD) for the case with and without rotating magnetic field, respectively and the related Ginzburg-Landau equations.

In this paper, we propose a new $hp$-adaptive finite-element strategy to solve the non-linear eigenvector problem corresponding to the critical points of the energy functional~\eqref{eq:GPenergy}.
Our approach relies on a combination of the metric-adaptive gradient flow~\cite{henning2020sobolev,henning2025convergence} and an energy-specific refinement strategy inspired from \cite{HeidStammWihler:21, HoustonWihler:16} (see also the related local error reduction approach~\cite{bammer2025hp} for $hp$--adaptive discretisations of general linear elliptic problems). The energy-based refinement strategy is considerably simpler compared to a residual-based a posteriori criteria. 
Further, it is important to note that within an adaptive strategy, the gradient descent can be combined with the adaptive procedure merging both iterative processes into a single unified one. 
With respect to~\cite{HeidStammWihler:21}, we introduce three key novel developments. First, we extend the GPE to the setting of a rotating magnetic field. Second, we establish a convergence result for a discrete gradient-flow iteration (Theorem~\ref{thm:discreteconvergence}) that is derived from an energy-related continuous projected Sobolev gradient flow; these ideas follow closely the work~\cite{henning2020sobolev}, where the rotation-free case has been considered. Third, we propose a new $hp$-adaptive strategy that yields substantially faster (exponential) convergence with respect to the number of degrees of freedom.

Mesh-adaptive strategies are particularly suited in this case, since the number and location of the vortices is unknown a priori. 
Furthermore, high-order approximations enable the computation of highly accurate ground-states and thus high-fidelity in the computed energies. 

The outline of this paper is as follows. In \textsection\ref{sec:gradient_flow}, we begin by specifying the problem setting and recalling the framework of the metric-adaptive Sobolev gradient approach; we also prove some instrumental results for the discrete gradient flow, cf.~Theorem~\ref{thm:discreteconvergence}. In \textsection\ref{sec:hpfem} we present the $hp$-adaptive strategy, while in \textsection\ref{sec:numerics} we conduct some numerical experiments illustrating the superior performance of the proposed energy-based $hp$-adaptive approach. Finally, in \textsection\ref{sec:conclusions} we summarise the work presented in this article and highlight potential future areas of research.

\section{Projected Sobolev gradient flow} \label{sec:gradient_flow}
In this section, we introduce the GPE for the associated minimisation problem given in~\eqref{eq:GPproblem} below, and establish the continuous and discrete projected Sobolev gradient flow for its exact and numerical solution, respectively. 

\subsection{Gross--Pitaevskii energy minimisation with external magnetic field.}
Throughout the paper, we make the following structural assumption on the external potential
$V \in \mathrm{L}^\infty(\Omega;\mathbb{R})$:
\begin{enumerate}[(V1)]
 \item We assume that 
 \begin{align} \label{eq:V1}
\delta_\Omega:=\essinf_{\x\in\Omega}\left( V(\x) - \frac12 \omega^2 |\x|^2\right)>0,
 \end{align}
 where $\x=(x,y) \in \mathbb{R}^2$ and $\x=(x,y,z) \in \mathbb{R}^3$, respectively, denote the Euclidean coordinates, and $|\cdot|$ the Euclidean norm in $\mathbb{R}^d$, for space dimensions~$d=2,3$.
\end{enumerate}
Since $\Omega$ is a bounded domain, we note that we can always shift the potential $V$ by a constant value so that (V1) is satisfied; indeed, this simply results in a corresponding shift of the energy. Finally, the term $ i\, \bb(\x) \cdot \nabla$ occurring in ~\eqref{eq:GPenergy}, with the real-valued vector function
\begin{equation}\label{eq:A}
\bb(\x):=
\begin{cases}
(y,-x)^\intercal&\text{for $\x=(x,y) \in \Omega\subset\mathbb{R}^2$}, \\
(y,-x,0)^\intercal&\text{for $\x=(x,y,z) \in \Omega\subset\mathbb{R}^3$},
\end{cases}
\end{equation}
takes the role of the angular momentum operator; without loss of generality, for $d=3$, we assume that the rotation takes place about the $z$-axis. Since $\bb$ is divergence-free, we note that
\[
\Div\left(\bb|u|^2\right)
=\bb\cdot\nabla \left(|u|^2\right)
=\bb\cdot\nabla(\overline{u}u)
=\overline{u}(\bb\cdot\nabla u)+u(\bb\cdot\nabla\overline{u});
\]
hence,
\[
\overline{u}(\bb\cdot\nabla u)=\Div\left(\bb|u|^2\right)-u(\bb\cdot\nabla\overline{u}).
\]
From this equality, upon application of the divergence theorem and recalling that $u=0$ on $\partial\Omega$, we infer the identity
\[
\int_\Omega \omega \overline{u} (\bb \cdot \nabla u) \dx=-\int_\Omega \omega u (\bb \cdot \nabla \overline{u}) \dx.
\]
Thereby, a straightforward calculation reveals that the Gross--Pitaevskii energy functional~\eqref{eq:GPenergy} can be stated equivalently as
\begin{align*}
\E(u)=\int_\Omega \left( \frac{1}{2}\left|-i \nabla u+\omega \bb u\right|^2+ \left(V(\x)-\frac{1}{2}\omega^2|\x|^2\right)|u|^2+\frac{\beta}{2}|u|^4 \right) \dx.
\end{align*}
In particular, in view of~\eqref{eq:V1}, this implies that the energy functional from~\eqref{eq:GPenergy} only takes non-negative real values, meaning that
\begin{align} \label{eq:realvalued}
\{\E(v):\,v\in\H\}\subset\mathbb{R}_{\ge 0}.
\end{align}

The focus of this paper is the numerical approximation of a global minimiser of the energy functional $\E$ on the sphere $\S_\H$, which we refer to as a \emph{ground-state} of the Gross-Pitaevskii energy~\eqref{eq:GPenergy}. The following result generalises the existence theorems of \cite{bao2005ground,bao2013mathematical}. 

\begin{proposition}
Under the assumption {\rm (V1)} above there exists a (non-unique) $u_\mathrm{GS}\in\S_\H$, which satisfies
\begin{align} \label{eq:GPproblem}
 \E(u_\mathrm{GS})= \min_{v \in \S_\H} \E(v)\ge 0;
\end{align}
i.e., in particular, the Gross-Potaevskii energy functional $\E$ from \eqref{eq:GPenergy} restricted to $\S_\H$ attains its minimum.
\end{proposition}

\begin{proof}
From~\eqref{eq:realvalued} we notice that the energy functional $\E$ is bounded from below. In addition, based on the condition~(V1), it is coercive (to be shown later on, see~\eqref{eq:Ecoercive} below). Therefore, there exists a bounded sequence $\{u_n\}_n \subset \S_{\H}$ with
\begin{align*}
\lim_{n \to \infty} \E(u_n)=\inf_{v \in \S_\H} \E(v).
\end{align*} 
By the reflexivity of $\H$ and exploiting the compact embedding $\H\hookrightarrow\L{2}$ we can extract a subsequence $\{u_{n_k}\}_k$ that converges weakly in $\H$ and strongly in $\L{2}$ to some element $u \in \H$. Due to the convergence in $\L{2}$, we indeed have that $\norm{u}_{\L{2}}=1$; i.e., $u \in \S_{\H}$. We further remark that the energy functional $\E$ is convex, which follows from the same arguments as in the proof of the coercivity from Proposition~\ref{prop:innerproduct} below, and strongly continuous. In turn, the functional $\E$ is weakly sequentially lower semicontinuous; see, e.g.,~\cite[Prop.~25.20]{Zeidler:90}. Consequently, we have that 
\[\E(u) \leq \liminf_{k \to \infty} \E(u_{n_k}) = \inf_{v \in \S_\H} \E(v),\]
thereby implying that $u$ is indeed a minimiser of the energy functional $\E$ on $\S_{\H}$.
\end{proof}

\subsection{The Gross--Pitaevskii equation}
In order to find a minimum of $\E$ under the $\L2$-normalisation constraint 
\begin{equation}\label{eq:L2constraint}
\Lnorm{u}^2:=(u,u)_{\L2}:=\int_\Omega u\overline u\dx=1,
\end{equation}
cf.~\eqref{eq:GPproblem}, we consider the Lagrange functional
\[
\mathsf{L}(u,\lambda):=\E(u)-\frac{\lambda}{2} \left(\int_\Omega |u|^2 \ \dx - 1\right).
\]
Then, considering $\H=\mathrm{H}^1_0(\Omega;\mathbb{C})$ as a \emph{real} vector space, any (local) minimiser $u$ of $\E$ restricted to $\S_\H$ satisfies the Euler--Lagrange equation
\begin{align}\label{eq:ELabstract}
\dprod{\E'(u),v}=\lambda \Re\left(\int_\Omega u \overline{v} \dx\right) \qquad \text{for all} \ v \in \H,
\end{align} 
where $\lambda \in\mathbb{R}$ is the Lagrange multiplier and $\E'$ denotes the \emph{real} G\^{a}teaux derivative of $\E$. More precisely, $\E':\H \to \Hdual$ is characterised by
\[
\dprod{\E'(u),v}=\left.\frac{{\d}}{\mathrm{dt}} \E(u+tv)\right|_{\R\ni t\to0}, \qquad u,v \in \H,
\]
with $\Hdual$ signifying the dual space of~$\H$ (composed of all bounded linear forms on $\H$ that map into~$\R$), and $\dprod{\cdot,\cdot}:\,\Hdual\times\H\to\R$ the associated dual product. 
A simple calculation employing the Riesz representation theorem reveals that
\begin{align}\label{eq:Egateaux}
\dprod{\E'(u),{v}}
= \Re \left(\int_\Omega\left( \nabla u \cdot \nabla \overline{v} + 2V(\x) u \overline{v} +2\beta |u|^2 u\overline{v}-2i\omega \overline{v} (\bb(\x) \cdot \nabla u) \right) \dx\right),
\end{align}
and thus, the Euler--Lagrange problem~\eqref{eq:ELabstract} consists in finding a pair $(\lambda,u) \in \mathbb{R} \times \H$ such that
\begin{align}\label{eq:RealEL}
\Re \left( \int_\Omega\left(\nabla u \cdot \nabla \overline{v} + 2 V(\x) u \overline{v} +2 \beta |u|^2 u\overline{v}-2i\omega \overline{v} (\bb(\x) \cdot \nabla u) \right) \dx \right) = \lambda \Re\left(\int_\Omega u \overline{v} \dx \right)
\end{align}
for all $v \in \H$. Equivalently, see~\cite[Prop.~A.1]{henning2025discrete}, a solution $(\lambda,u)$ of~\eqref{eq:RealEL} satisfies
\begin{align} 
\label{eq:GPEweak}
 \int_\Omega\left(\nabla u \cdot \nabla \overline{v} + 2 V(\x) u \overline{v} +2 \beta |u|^2 u\overline{v}-2 i\omega \overline{v} (\bb(\x) \cdot \nabla u) \right) \dx
 =
 \lambda (u,v)_{\L2} \qquad \forall v \in \H,
\end{align}
meaning that it does not matter whether or not the real part is applied.  This is the weak form of the stationary \emph{GPE with angular momentum}, which can also be written as a semi-linear eigenfunction problem, namely, in strong form $(\lambda,u) \in \mathbb{R} \times \mathrm{C}_0^2(\Omega)$ satisfies
\begin{align} \label{eq:GPE}
 - \Delta u + 2V(\x) u + 2\beta |u|^2 u - 2i \omega (\bb(\x) \cdot \nabla u)= \lambda u. 
\end{align} 
An $\L{2}$-normalised solution $u \in \H$ of~\eqref{eq:GPE} with \emph{minimal energy} is termed a \emph{ground-state}, denoted by $u_\mathrm{GS}\equiv u$, cf.~\eqref{eq:GPproblem}, whereas any $\L{2}$-normalised $u\in\H$ solving~\eqref{eq:GPE}, with an energy strictly higher than the ground-state energy, i.e.,
$
\E(u) > \E(u_\mathrm{GS})= \min_{v \in \S_\H}\E(v),
$
is called an \emph{excited state}.  

\begin{remark}
The idea of using the real vector space $\H$, and, in turn, the real G\^{a}teaux derivative of $\E:\H \to \mathbb{R}$ for the formulation of the Euler--Lagrange equation~\eqref{eq:ELabstract}, is standard practice in physics. Indeed, this is due to the fact that the mapping $t \mapsto \E(u+tv)$, for given $u,v \in \H$, takes real values only, even if $t$ were complex, cf.~\eqref{eq:realvalued}. Consequently, by the open mapping theorem, the energy functional $\E$ is not complex differentiable. Nonetheless, using the Wirtinger calculus, the weak formulation~\eqref{eq:GPEweak}
 can be derived in an alternative way as the Euler--Lagrange equation of the constraint minimisation problem~\eqref{eq:GPproblem}. 
 \end{remark}

\subsection{Continuous projected Sobolev gradient flow}

In this section, we aim to establish a projected Sobolev gradient flow formulation for the minimisation of the Gross--Pitaevskii energy functional $\E$ from~\eqref{eq:GPenergy}. For this purpose, we will consider the real framework of the GPE~\eqref{eq:GPEweak} rather than the complex Hilbert space approach. We emphasise, however, that both avenues ultimately result in the identical gradient flow. 

\subsubsection{Weighted energy inner-product}

In the spirit of~\cite{henning2020sobolev,henning2025convergence,henning2025discrete}, for given $z \in \H$, we first define the sesquilinear form
\begin{align} \label{eq:innerproduct}
 a_z(u,v):=\int_\Omega \left( \nabla u \cdot \nabla \overline{v} + 2 V(\x) u \overline{v} + 2 \beta |z|^2 u \overline{v} - 2 i \omega \overline{v} \left(\bb(\x) \cdot \nabla u \right) \right)\dx, \qquad u,v \in \H,
\end{align}
cf.~\cite[Eq.~(7)]{henning2025convergence}, which will be investigated in the ensuing result. 

\begin{proposition}\label{prop:innerproduct}
For any given $z \in \H$, the map $a_z: \H \times \H \to \mathbb{C}$ from \eqref{eq:innerproduct} defines a (complex) inner-product on the space $\H=\mathrm{H}_0^1(\Omega;\mathbb{C})$. Moreover, the induced norm is equivalent to the $\H$-norm defined in~\eqref{eq:Hnorm}; especially, for a constant $c>0$ independent of $z \in \H$, the uniform coercivity property $a_z(u,u) \geq c \norm{\nabla u}_{\L2}^2$ for all $u \in \H$ holds true.  
\end{proposition}

\begin{proof}
We will recall several ideas from \cite[\textsection 2.1]{danaila2010new}. Throughout, let $z \in \H$ be arbitrary. 
\begin{enumerate}[(i)]
\item \emph{Bilinearity and symmetry.} First, we note that the mapping $u \mapsto a_z(u,v)$ is linear for any fixed $v \in \H$. Moreover, we have that $a_z(u,v)=\overline{a_z(v,u)}$ for any $u,v \in \H$. Indeed, since $\beta \in \mathbb{R}$ and $V \in \mathrm{L}^\infty(\Omega;\mathbb{R})$, the conjugate-symmetry is directly observed for the first three terms in the integral from~\eqref{eq:innerproduct}. Moreover, for the last term, we exploit the fact that $u$ and $v$ vanish along the boundary of $\Omega$, and hence the application of the divergence theorem implies that
\[
  - i \omega \int_\Omega \overline{v} (\bb \cdot \nabla u) \dx =  i \omega \int_\Omega u \Div (\overline{v} \bb) \dx = i \omega \int_\Omega u (\nabla \overline{v} \cdot \bb)  \dx = \overline{-i\omega \int_\Omega \overline{u} (\bb\cdot \nabla v) \dx},
  \]
since $\bb$ is real-valued and divergence-free. 

\item \emph{Coercivity and equivalence to the $\H$-norm.} We first note that
\begin{align}
a_z(u,u) \leq \int_\Omega \left( |\nabla u|^2+(2V+2\beta |z|^2)|u|^2+2 \omega r_{\Omega} |\nabla u| |u|\right) \dx,\label{eq:azuu}
\end{align} 
with
\[
r_{\Omega}:=\sup_{\x \in \Omega}|\bb(\x)|=\sup_{\x \in \Omega}\sqrt{x^2+y^2},
\]
cf.~\eqref{eq:A}, with $\x=(x,y)\in\mathbb{R}^2$ or $\x=(x,y,z)\in\mathbb{R}^3$. Using the Cauchy-Schwarz inequality, the Rellich--Kondrachov embedding theorem, and the Poincar\'e-Friedrichs inequality, we observe that
\[
\int_\Omega|z|^2|u|^2\dx
\le\|z\|^2_{\mathrm{L}^4(\Omega)}\|u\|^2_{\mathrm{L}^4(\Omega)}
\le C\|\nabla u\|^2_{\Ltwo}\|\nabla z\|^2_{\Ltwo},
\]
for a constant $C>0$ which only depends on~$\Omega$. The remaining terms on the right-hand side of~\eqref{eq:azuu} can be estimated in an analogous way, so that the upper bound
$
a_z(u,u) \leq C \Lnorm{\nabla u}^2$ is obtained
for all $u \in \H$,
where the constant $C > 0$ depends on $\norm{V}_{\L{\infty}}$, $\|\nabla z\|_{\Ltwo}$, $\beta$, $\omega$, and $\Omega$. In order to derive the lower bound, for any $\varepsilon\in(0,1)$ and $\vec{x}\in\Omega$, employing the arithmetic-geometric mean inequality gives
\begin{align}\label{eq:ibound}
2\left|\omega \overline{u} (\bb(\x) \cdot \nabla u) \right| \leq  \varepsilon |\nabla u|^2+\varepsilon^{-1}\omega^2 r^2_\Omega|u|^2.
\end{align}
Selecting 
\begin{equation}\label{eq:eps}
\varepsilon:=\frac{\omega^2 r_{\Omega}^2}{\omega^2 r_{\Omega}^2+2 \delta_\Omega} \in (0,1),
\end{equation}
with $\delta_\Omega>0$ from~\eqref{eq:V1}, it follows that 
\begin{equation}\label{eq:eps1}
2V(\x)-\varepsilon^{-1}\omega^2r^2_\Omega \geq 0,\qquad\text{for (almost) all }\vec{x}\in\Omega.
\end{equation}
Exploiting that $\beta \geq 0$, we have
\begin{align}\label{eq:c1}
c_1(\x)
:=&(2V(\x)+2\beta|z|^2) -  \varepsilon^{-1}\omega^2 r^2_\Omega
\ge0,\qquad \x \in \Omega.
\end{align}
Therefore, making use of~\eqref{eq:ibound} and~\eqref{eq:c1} leads to
\begin{align} \label{eq:coercivity}
a_z(u,u) \geq \int_\Omega (1-\varepsilon)|\nabla u|^2+c_1(\x) |u|^2 \dx \geq c \norm{\nabla u}^2_{\L2}, 
\end{align}
where the constant is given by
\begin{align*} 
c:=1-\varepsilon=\frac{2 \delta_\Omega}{\omega^2 r_\Omega^2+2 \delta_\Omega}>0;
\end{align*}
cf.~\eqref{eq:V1}.

\item Finally, the \emph{positive definiteness} of $a_z$ follows directly from~\eqref{eq:coercivity}.
\end{enumerate}
This complete the proof.
\end{proof}

Upon taking the real part, we immediately obtain a real inner-product.

\begin{corollary} \label{cor:realinner}
The bilinear form
\begin{align} \label{eq:realip}
(u,v) \mapsto \Re \left(a_z(u,v)\right), \qquad u,v \in \H,
\end{align}
with $a_z(\cdot,\cdot)$ from~\eqref{eq:innerproduct}, defines an inner-product on the \emph{real} Hilbert space $\H = \mathrm{H}_0^1(\Omega;\mathbb{C}).$
\end{corollary}

\begin{remark} \label{rem:GPEIP}
By definition of the inner-product $a_u(\cdot,\cdot)$, cf.~\eqref{eq:innerproduct}, the Gross--Piteavskii eigenvalue problem can be stated equivalently as follows: find $(\lambda,u) \in \mathbb{R} \times \H$ such that
\begin{align*}
\Re \left(a_u(u,v)\right)=\lambda \Re \left((u,v)_{\L{2}}\right) \qquad \Longleftrightarrow \qquad a_u(u,v)=\lambda(u,v)_{\L{2}}
\end{align*}
for all $v \in \H$.
\end{remark}

\begin{remark}
The arguments in the proof of Proposition~\ref{prop:innerproduct} can be applied to show that the energy functional $\E$ is coercive on $\H$. Indeed, applying the bound~\eqref{eq:ibound}, with $\varepsilon\in(0,1)$ from~\eqref{eq:eps}, we notice that
\begin{align*}
\E(u)
&\ge\int_\Omega \frac{1}{2}\left(\left(1-\varepsilon\right) | \nabla u|^2+\left(V(\x)-\frac12\varepsilon^{-1}\omega^2 r^2_\Omega\right)|u|^2+\frac{\beta}{2} |u|^4\right) \dx\qquad\forall u\in\H.
\end{align*}
Recalling~\eqref{eq:eps1} and the fact that $\beta\ge 0$, it follows that
\begin{equation}\label{eq:Ecoercive}
\E(u)\ge c\norm{\nabla u}^2_{\Ltwo}\qquad\forall u\in\H,
\end{equation}
with the constant $c=\nicefrac{(1-\varepsilon)}{2}>0$.
\end{remark}

\subsubsection{Energy-based orthogonal projection}

The idea of the Sobolev gradient flow to be applied in this work is to follow a trajectory in the direction of the steepest descent with respect to the topology induced by the real inner-product from~\eqref{eq:realip}. More precisely, for fixed $u \in \H$, we define the gradient descent direction $\nabla_{a_u}\E(u)$ of $\E$ at $u \in \H$ by
\begin{align} \label{eq:descentdirection}
\Re \left(a_u(\nabla_{a_u}\E(u),v)\right)=\dprod{\E'(u),v} \qquad \text{for all} \ v \in \H;
\end{align}
we emphasise that such an element $\nabla_{a_u} \E(u) \in \H$ exists and is unique thanks to Corollary~\ref{cor:realinner} and the Riesz representation theorem. In light of~\eqref{eq:Egateaux} and \eqref{eq:innerproduct}, we observe that
\begin{align} \label{eq:IPvsEG}
\Re \left(a_u(u,v)\right)=\dprod{\E'(u),v} \qquad \text{for any} \ u,v \in \H,
\end{align} 
and thus
\begin{align*}
\Re \left(a_u(\nabla_{a_u} \E(u),v) \right)= \Re \left(a_u(u,v)\right).
\end{align*} 
This, in turn, implies the identity 
\begin{equation}\label{eq:idGradE}
\nabla_{a_u} \E(u)=u\qquad\text{on } \H. 
\end{equation}
We note that we could omit taking the real-part on the left-hand side of~\eqref{eq:descentdirection}, if, instead of employing the real G\^{a}teaux derivative $\E'$ of $\E$ at $u$, the Wirtinger derivative were considered. This, however, would render the ensuing analysis unnecessarily more technical.

To incorporate the $\L2$-normalisation constraint \eqref{eq:L2constraint} in the dynamical system to be specified below, we project the gradient $\nabla_{a_u} \E(u)=u$ onto the linear tangent space associated to the sphere $\S_{\H}$ from~\eqref{eq:SH} at~$u$ given by 
\[
\mathbb{T}_{u}:=\{w \in \H: (w,u)_{\L2}=0\} =\left\{w \in \H : \Re\left( (w,u)_{\L2}  \right)= 0\right\}.
\]
Let $\mathsf{P}_{u}: \H \to \mathbb{T}_{u}$ be the orthogonal projection with respect to the inner-product $a_u(\cdot,\cdot)$, or equivalently to $\Re\left(a_u(\cdot,\cdot)\right)$; in particular, for any $v\in\H$, we have
\begin{equation}\label{eq:Pudef}
a_u(\mathsf{P}_{u}(v),w)=a_u(v,w)\qquad \forall w\in\mathbb{T}_u.
\end{equation}
For any $u\neq 0$, we note the explicit formula
\begin{equation}\label{eq:P}
v\mapsto\mathsf{P}_{u}(v):=v-\frac{(v,u)_{{\rm L}^2(\Omega)}}{(\mathsf{G}(u),u)_{\L2}}\mathsf{G}(u),
\end{equation}
where, for any $u \in \H$, the Riesz representer $\mathsf{G}(u) \in \H$ is uniquely defined by
\begin{equation}\label{eq:G}
a_u(\mathsf{G}(u),v)=(u,v)_{\Ltwo} \qquad \forall v \in \H,
\end{equation}
or equivalently by
\begin{equation} \label{eq:Greal}
  \Re\left(a_u(\mathsf{G}(u),v)\right)=\Re\left((u,v)_{\L2}\right) \qquad \forall v \in \H.
\end{equation}

\subsubsection{Projected gradient flow}

We introduce the dynamical system 
\begin{align} \label{eq:dynamicalsystem}
\dot{u}(t)=-\mathsf{P}_{u(t)}(\nabla_{a_{u(t)}} \E(u(t))) =-\mathsf{P}_{u(t)}(u(t))\quad\text{for } t > 0\quad\text{and}\quad u(0)=u^0,
\end{align}
for some initial guess $u^0 \in \S_\H$, where we have employed the identity~\eqref{eq:idGradE} in the second equality. Using the representation~\eqref{eq:P}, the above dynamical system~\eqref{eq:dynamicalsystem} can also be expressed by
\begin{subequations}\label{eq:continuousflow2}
\begin{align}
\dot{u}(t)=-\mathsf{P}_{u(t)}(u(t))&=-u(t)+\gamma_{u(t)}\mathsf{G}(u(t)) ,\quad t>0, \qquad u(0)=u^0,
\end{align}
with
\[
\gamma_{u(t)}=\frac{\|u(t)\|^2_{\L2}}{(\mathsf{G}(u(t)),u(t))_{\L2}},
\]
for some initial $u(0)=u^0 \in \S_\H$. In light of~\eqref{eq:G} and the coercivity of $a_{u(t)}$, cf.~Proposition~\ref{prop:innerproduct}, we observe that
\[
(\G(u(t)),u(t))_{\Ltwo} = \overline{(u(t),\G(u(t)))}_{\Ltwo} = a_{u(t)}(\G(u(t)),\G(u(t))) > 0;
\]
i.e., $\gamma_{u(t)}>0$ for any $t>0$. Furthermore, since 
$\dot{u}(t)=-\mathsf{P}_{u(t)}(u(t)) \in \mathbb{T}_{u(t)}$,
a straightforward calculation yields that
\[
\frac{{\d}}{{\d}t} \Lnorm{u(t)}^2=2 \Re\left((u(t),\dot{u}(t))_{\L2}\right)=0,
\]
i.e. the flow preserves the $\Ltwo$-norm. Consequently, for $u_0\in\S_\H$, we conclude that
\begin{equation}
    \gamma_{u(t)}=\frac{1}{(\mathsf{G}(u(t)),u(t))_{\L2}}>0.
\end{equation}
\end{subequations}
Finally, we remark that the energy $\E$ is monotone decreasing along the trajectory $u(t)$ as~$t$ progresses. Indeed, letting $g(t):=\E(u(t))$, $t \ge0$,
we invoke~\eqref{eq:IPvsEG} to obtain
\[
g'(t)
= 
\dprod{\E'(u(t)),\dot{u}(t)}
=
\Re \left( a_{u(t)}(u(t),\dot{u}(t)) \right).
\]
Then, combining~\eqref{eq:Pudef} and~\eqref{eq:continuousflow2}, we deduce that
\[
g'= \Re(a_u(\mathsf{P}_{u}(u),\dot{u}))= -\Re(a_u(\dot{u},\dot{u}))
=-a_u(\dot{u},\dot{u})
\leq 0,
\]
cf.~Proposition~\ref{prop:innerproduct}; i.e., the energy $\E(u(t))$ is monotonically decreasing in~$t$.

\subsection{Discrete gradient flow}\label{sc:discflow}

For the purpose of computing an approximation of the continuous projected Sobolev gradient flow trajectory from~\eqref{eq:continuousflow2}, we employ the forward Euler time stepping method. Specifically, for a given initial value $u^0 \in \S_\H$, this yields a sequence $\{u^n\}_{n\ge 0}\subset \S_\H$, which, for~$n\ge0$, is defined by
\begin{subequations}\label{eq:GFiteration}
\begin{align} 
u^{n+1}&=\frac{\widehat{u}^{n+1}}{\Lnorm{\widehat{u}^{n+1}}},\label{eq:GFiteration1}
  \intertext{where}
  \widehat{u}^{n+1}&
  =u^n-\tau_n\mathsf{P}_{u^n}(u^n)
  =(1-\tau_n)u^n+\frac{\tau_n}{(\mathsf{G}(u^n),u^n)_{\L2}}\mathsf{G}(u^n).\label{eq:GFiteration2}
\end{align}
\end{subequations}
Here, $\{\tau_n\}_{n\ge 0}$ is a sequence of positive (discrete) time steps that is assumed to be uniformly bounded from above and below with bounds~$\tau_{\max}$ and~$\tau_{\min}$, respectively, such that
\begin{equation}\label{eq:timesteps}
0 < \tau_{\min} \leq \tau_n \leq \tau_{\max} < 2\qquad\forall n\ge 0.
\end{equation}
The scheme~\eqref{eq:GFiteration} is called \emph{discrete gradient flow iteration (GFI)}. We note that this iteration coincides with the one from \cite{henning2025convergence}, and, up to an additional term in the weighted inner-product that incorporates the angular momentum, resembles the GFI proposed in~\cite[\textsection 4]{henning2020sobolev}. In particular, we may borrow the following result from \cite[Thm.~3.4]{henning2025convergence} in a modified form, which can also (largely) be obtained by following along the lines of~\cite[\textsection 4]{henning2020sobolev} (with some minor adaptations).

\begin{theorem} \label{thm:discreteconvergence}
For a sufficiently small lower bound $\tau_{\min}$, there exists a constant $\tau_{\max} \in (\tau_{\min},2)$ such that, for any sequence of time-steps $\{\tau_n\}_n \subset [\tau_{\min},\tau_{\max}]$, the resulting sequence $\{u^n\}_n$ generated by the GFI~\eqref{eq:GFiteration} satisfies the following properties:
\begin{enumerate}[\rm(a)]
\item We have that $\E(u^{n+1}) \leq \E(u^n)$ for all $n \geq 0$. More precisely, there exists a constant $C_\tau>0$ (only depending on $\tau_{\min}$ and $\tau_{\max}$) such that
\begin{align} \label{eq:Econ}
\E(u^{n})-\E(u^{n+1}) \geq C_\tau \Lnorm{u^{n+1}-u^n}^2
\end{align}
for all $n \geq 0$.
\item The limit $\E^\star:=\lim_{n \to \infty} \E(u^n)$
is well-defined. Moreover, there exists a strongly convergent subsequence $\{u^{n_k}\}_k$ in $\H$ with a limit $u^\star \in \S_\H$, and $\E(u^\star)=\E^\star$. 

\item Furthermore, for the limit $u^\star$ from {\rm (b)}, and upon letting 
\begin{align*}
\lambda^\star:=\frac{1}{(\mathsf{G}_{u^\star}(u^\star),u^\star)_{\L2}},
\end{align*}
we have that
\begin{align}\label{eq:u*}
a_{u^\star}(u^\star,v)
=
\lambda^\star(u^\star,v)_{\L2} 
\qquad \forall v \in \H,
\end{align}
i.e. $(\lambda^\star,u^\star)$ is a (weak) solution to the Gross--Pitaevskii EVP~\eqref{eq:GPE}.   
\end{enumerate}
\end{theorem}

\begin{proof}
We will sketch the proof by focussing on the required modifications of the techniques presented in~\cite[\textsection 4]{henning2020sobolev} for the more general case including angular momentum. We proceed in several steps. 

\begin{enumerate}[(i)]
\item A straightforward calculation reveals that the mass, without the scaling~\eqref{eq:GFiteration1}, grows in the iteration step~\eqref{eq:GFiteration2}. To this end, we can follow (exactly) along the lines of\cite[\textsection 4.1]{henning2020sobolev}. In particular, upon exploiting~\eqref{eq:GFiteration2}, i.e.,
\[
\widehat{u}^{n+1}-u^n
  =-\tau_n\mathsf{P}_{u^n}(u^n)
  =-\tau_nu^n+\frac{\tau_n}{(\mathsf{G}(u^n),u^n)_{\L2}}\mathsf{G}(u^n),
\]
we obtain that $(u^n,\widehat{u}^{n+1}-u^n)_{\L2}=0$,
and thus
\begin{align} \label{eq:massgrowth}
\norm{\widehat{u}^{n+1}}^2_{\L2} = \norm{u^n}^2_{\L2}+\norm{\widehat{u}^{n+1}-u^n}_{\L2}^2 \geq \norm{u^n}^2_{\L2}.
\end{align}

\item Next, we will verify the (discrete) energy dissipation. This is done in several steps and will be summarised in the following. Since the energy, in addition to the angular momentum term, is slightly different from~\cite{henning2020sobolev}, we will provide some more details. Firstly, using~\eqref{eq:G} and applying a similar argument as in (i), we observe that
$
a_{u^n}(\G(u^n),\widehat{u}^{n+1}-u^n)=0,
$
and
\begin{equation}\label{eq:aorth}
a_{u^n}(\widehat{u}^{n+1}-u^n,\widehat{u}^{n+1}-u^n)
=-\tau_n a_{u^n}(u^n,\widehat{u}^{n+1}-u^n).
\end{equation}
Equivalently,
\begin{align} \label{eq:E}
a_{u^n}(u^n,u^n)=a_{u^n}(\widehat{u}^{n+1},\widehat{u}^{n+1})+\left(\nicefrac{2}{\tau_n} -1 \right) a_{u^n}(\widehat{u}^{n+1}-u^n,\widehat{u}^{n+1}-u^n),
\end{align}
cf.~the proof of~\cite[Lem.~4.5]{henning2020sobolev}. Moreover, invoking~\eqref{eq:GPenergy} and~\eqref{eq:innerproduct}, we find that
\begin{align*}
a_{u^n}(u^n,u^n) = 2 \E(u^n) + \beta \int_\Omega |u^n|^4 \dx,
\end{align*}
and similarly
\begin{align}\label{eq:Gh}
a_{u^n}(\widehat{u}^{n+1},\widehat{u}^{n+1})=2 \E(\widehat{u}^{n+1})+ 2 \beta \int_\Omega |u^n|^2|\widehat{u}^{n+1}|^2\dx - \beta \int_\Omega |\widehat{u}^{n+1}|^4 \dx.
\end{align}
Combining the identities~\eqref{eq:E} and \eqref{eq:Gh}, we deduce that
\begin{align*}
\E(u^n)&-\E(\widehat{u}^{n+1})\\
&= - \frac{\beta}{2} \int_\Omega \left(|u^n|^2-|\widehat{u}^{n+1}|^2\right)^2 \dx +\left(\frac{1}{\tau_n}-\frac{1}{2}\right)a_{u^n}(\widehat{u}^{n+1}-u^n,\widehat{u}^{n+1}-u^n).
\end{align*}
Furthermore, repeating the calculations in the proof of Theorem~\cite[Lem.~4.5]{henning2020sobolev} (with $\beta$ replaced by $2\beta$), we arrive at
\begin{align*}
\E(u^n)-\E(\widehat{u}^{n+1}) &\geq \left(\frac{1}{\tau_n}-\frac{1}{2} \right) a_0(\widehat{u}^{n+1}-u^n,\widehat{u}^{n+1}-u^n) -\beta \int_\Omega |\widehat{u}^{n+1}-u^n|^4 \dx \\
& \quad + 2\beta \left(\frac{1}{\tau_n}-\frac{5}{2} \right)\int_\Omega|\widehat{u}^{n+1}-u^n|^2|u^n|^2\dx. 
\end{align*}
In turn, for $\tau_n \leq \nicefrac{2}{5}$, we have that
\begin{align} \label{eq:ediffbound}
\E(u^n)-\E(\widehat{u}^{n+1}) \geq \left(\frac{1}{\tau_n}-\frac{1}{2} \right) a_0(\widehat{u}^{n+1}-u^n,\widehat{u}^{n+1}-u^n) -\beta \int_\Omega |\widehat{u}^{n+1}-u^n|^4 \dx. 
\end{align}

\item Based on~\eqref{eq:massgrowth}, \eqref{eq:aorth} and~\eqref{eq:ediffbound}, it can be shown by an inductive argument, cf.~the proof of~\cite[Lem.~4.7]{henning2020sobolev}, that there exists $0<\tau_{\max} \leq \nicefrac{2}{5}$ (depending on $\E(u^0)$ and $\beta$) such that the energy decays. Specifically, similarly to~\cite[Eq.~(26)]{henning2020sobolev}, we derive the bound
\begin{align} \label{eq:diffbound2}
\E(u^n)-\E(\widehat{u}^{n+1}) \geq c\, a_0(\widehat{u}^{n+1}-u^n,\widehat{u}^{n+1}-u^n)\ge 0,
\end{align}
for some $c>0$ independent of $n$ (but depending on $\tau_{\max}$ and $ \, \beta$). Then,  
\begin{align} \label{eq:energydissipation}
\E(u^{n+1}) \leq \E(\widehat{u}^{n+1}) \leq \E(u^n),
\end{align}
where the first inequality follows from the mass growth obtained in~(i). This proves the first part of (a). Concerning the bound~\eqref{eq:Econ}, we simply refer to the proof in~\cite[Thm~3.4]{henning2025convergence}.\smallskip

\item By the energy decay~\eqref{eq:energydissipation} and the non-negativity of $\E$, cf.~\eqref{eq:realvalued}, it immediately follows that the sequence $\{\E(u^n)\}_{n}$ has a limit $\E^\star\ge 0$. In turn, by virtue of the coercivity bound~\eqref{eq:Ecoercive}, the sequence $\{u^n\}_n$ is uniformly bounded in $\H$, and thus has a weakly convergent subsequence $\{u^{n_j}\}_j$ in $\H$ with a limit $u^\star \in\H$; thanks to the Rellich-Kondrachov compactness theorem we may assume that the convergence is strong in $\mathrm{L}^p(\Omega;\mathbb{C})$, for $1 \leq p < 6$ in dimensions $d\in\{2,3\}$; in particular, we have $u^\star\in\S_{\H}$. Moreover, using~\eqref{eq:diffbound2} and \eqref{eq:energydissipation}, and employing Proposition~\ref{prop:innerproduct}, we obtain the bounds
\[
\E(u^n)-\E(u^{n+1})
\ge \E(u^n)-\E(\widehat{u}^{n+1})\ge \widetilde c \norm{\widehat{u}^{n+1}-u^n}^2_{\H},
\]
for a constant $\widetilde c>0$. Hence, from the convergence of the energy, we further deduce that
$\norm{\widehat{u}^{n+1}-u^n}_{\H} \to 0$. Then, by repeating the calculations in~\cite[Thm.~4.9]{henning2020sobolev}, we find that
\begin{align} \label{eq:lambdaconv}
\nicefrac{1}{\lambda^{n_j}}:=(\G(u^{n_j}),u^{n_j})_{\L2} \to (\G(u^\star),u^\star)_{\Ltwo}=:\nicefrac{1}{\lambda^\star} \quad \text{as} \ j \to \infty,
\end{align}
from which~\eqref{eq:u*} can be obtained. In particular, the claim in (c) is verified.

\item Next, we will show that $\E(u^\star)=\E^\star$. Our proof is slightly different from the one in~\cite{henning2020sobolev}. First of all, using the definitions of the energy functional  and of the weighted inner-product from~\eqref{eq:GPenergy} and~\eqref{eq:innerproduct}, respectively, we find that
\begin{align*}
2\left|\E(u^\star)-\E(u^{n_j})\right| &= \left| a_{u^\star}(u^\star,u^\star)-\beta \int_\Omega |u^\star|^4 \dx-a_{u^{n_j}}(\unj,\unj)+\beta \int_\Omega |\unj|^4 \dx\right| \\
& \leq |\lambda^\star-\lambda^{n_j}|+|\lambda^{n_j}-a_{\unj}(\unj,\unj)|+\beta \left|\Lfnorm{u^\star}^4-\Lfnorm{\unj}^4\right|;
\end{align*}
the first term vanishes in the limit $j \to \infty$ thanks to~\eqref{eq:lambdaconv}, while the last term vanishes as $j \to \infty$ due to the strong convergence in $\mathrm{L}^4(\Omega)$.  Concerning the second summand, in light of~\eqref{eq:GFiteration2}, we first observe that
\[
\frac{1}{\tau_{n_j}}a_{\unj}(\widehat{u}^{n_j+1},\unj)-\frac{1-\tau_{n_j}}{\tau_{n_j}}a_{\unj}(\unj,\unj) = \lambda^{n_j},
\]
and in turn
\[
|\lambda^{n_j}-a_{\unj}(\unj,\unj)|=\frac{1}{\tau_{n_j}}\left|a_{\unj}(\widehat{u}^{n_j+1}-\unj,\unj)\right|.
\]
Making use of~\eqref{eq:aorth} yields
\begin{align*}
|\lambda^{n_j}-a_{\unj}(\unj,\unj)| &= \tau^{-2}_{n_j} a_{\unj}(\widehat{u}^{n_j+1}-\unj,\widehat{u}^{n_j+1}-\unj) 
\lesssim \tau_{\min}^{-2} \norm{\widehat{u}^{n_j+1}-\unj}^{2}_{\H};
\end{align*}
here, we employed Proposition~\ref{prop:innerproduct} and the uniform boundedness of the sequence $\{u^{n_j}\}_j$. Furthermore, recalling the strong convergence $\norm{\widehat{u}^{n+1}-u^n}_{\H} \to 0$ as $n \to \infty$ from (iv), we conclude that the second summand vanishes as well. Hence, $\lim_{j\to\infty}\left|\E(u^\star)-\E(u^{n_j})\right|=0$.

\item It remains to verify that the subsequence $\{u^{n_j}\}_j$ converges to $u^\star$ strongly in $\H$. Since the proof of this step in the present work is more involved than in the case without angular momentum, we shall provide full details. From~(iv) and~(v), we recall that $\E(u^{n_j}) \to \E(u^\star)$ and $u^{n_j} \to u^\star$ strongly in $\mathrm{L}^{p}(\Omega)$, for $1 \leq p <6$, and $u^{n_j} \rightharpoonup u^\star$ weakly in $\H$. Thereby, this implies that
\begin{align*} 
\lim_{j \to \infty} \left( \frac{1}{2}\norm{u^{n_j}}^{2}_{\H}-\int_\Omega i \omega \overline{u^{n_j}}(\bb \cdot \nabla u^{n_j}) \dx \right) = \frac{1}{2} \norm{u^\star}^{2}_\H -\int_\Omega i \omega \overline{u^\star} (\bb \cdot \nabla u^\star) \dx.
\end{align*}   
In particular, if we can show that
\begin{align} \label{eq:helpnc}
\lim_{j \to \infty} \int_\Omega \overline{u^{n_j}}(\bb \cdot \nabla u^{n_j}) \dx = \int_\Omega \overline{u^\star} (\bb \cdot \nabla u^\star) \dx,
\end{align}
then $\norm{u^{n_j}}_\H \to \norm{u^\star}_\H$ as $j \to \infty$, which in combination with the weak convergence in $\H$ implies the strong convergence. Hence, it remains to verify~\eqref{eq:helpnc}. A simple manipulation reveals that
\begin{align*}
&\left|\int_\Omega \overline{u^{n_j}}(\bb \cdot \nabla u^{n_j}) \dx - \int_\Omega \overline{u^\star} (\bb \cdot \nabla u^\star) \dx \right| \\
& \leq \int_\Omega \left|u^{n_j}-u^\star\right| \left|\bb \cdot \nabla u^{n_j}\right| \dx + \left|\int_\Omega \left(\overline{u^\star}(\bb \cdot \nabla u^{n_j})-\overline{u^\star}(\bb \cdot \nabla u^\star) \right)\dx \right| \\
&=:I_1+I_2.
\end{align*}
For the first term, we recall that $\{u^{n_j}\}$ is a weakly convergent sequence in $\H$, and thus is uniformly bounded. Therefore, the Cauchy--Schwarz inequality yields that
\begin{align*}
I_1 \leq C \norm{u^{n_j}-u^\star}_{\L2} \to 0 \quad \text{as} \ j \to \infty,
\end{align*}
for some constant $C>0$ independent of $j$.
To deal with the second term, we note that $v \mapsto \int_\Omega \overline{u^\star} (\bb \cdot \nabla v) \dx$ is a bounded linear operator on $\H$, and thus $I_2 \to 0$ as $j \to \infty$ thanks to the weak convergence of the sequence.

\end{enumerate}
This completes the proof.
\end{proof}

\section{$hp$--Adaptive Gradient Flow Finite Element Discretisation} \label{sec:hpfem}

\subsection{$hp$-finite element discretisation}

For the purpose of a computational realisation of the discrete GFI~\eqref{eq:GFiteration}, Galerkin spaces $\X_N$ will be constructed in terms of an $hp$--finite element approach. We focus on a sequence of conforming and shape-regular partitions $\{\T_N\}_{N\in\mathbb{N}}$ of the domain~$\Omega$ into simplicial elements~$\T_N=\{\kappa\}_{\kappa\in\T_N}$ (i.e.,~triangles for $d=2$ and tetrahedra for~$d=3$), although more general elements could also be considered. For $N\ge 0$ and an element $\kappa\in\T_N$, we associate a polynomial degree~$p_\kappa \geq 1$ and collect these quantities
 in the polynomial degree vector~$\p_N=[p_\kappa:\kappa\in\T_{N}]$. With
this notation, for any subset~$\patch \subset\T_N$, we introduce the following $hp$--finite element space:
\begin{equation}\label{eq:locspace}
\V(\patch,\p_N):=\left\{v\in \H:\,
v|_{\kappa}\in\mathbb{P}_{p_\kappa}(\kappa),\kappa\in\patch,\,v|_{\Omega\setminus\patch}=0 \right\}, 
\end{equation}
where, for~$p\ge 1$, we denote by~$\mathbb{P}_p(\kappa)$ the local space of
polynomials of degree at most~$p$ on~$\kappa$. Furthermore, similarly as before, we denote by
\[
\S_{\V(\patch,\p_N)}=\left\{v\in\V(\patch,\p_N):\,\|v\|_{\L2}=1\right\}
\]
the $\L2$-unit sphere in~$\V(\patch,\p_N)$. In the sequel, for simplicity of notation, we write $\X_N:=\V(\mathcal{T}_N,\p_N)$ and $\S_N:=\S_{\V(\mathcal{T}_N,\p_N)}$.
Furthermore, the restriction of the energy functional $\E$ from \eqref{eq:GPenergy} to the Galerkin space $\X_N$ is signified by $\E_N:=\E|_{\X_N}$. Then, due to the compactness of the unit sphere $\S_N$ in discrete spaces, there exists a (possibly non-unique) minimiser $u_{N} \in \X_N$ of $\E_N$, i.e., $\E_N(u_N)=\min_{v \in \S_{N}} \E_N(v)$, with 
\begin{equation}\label{eq:unorm}
(u_{N},1)_{{\rm L}^2(\Omega)} \geq 0
\end{equation} 
that satisfies the discrete GPE
\begin{align} \label{eq:GPEdiscrete}
\begin{split}
\int_\Omega\big( \nabla u_N \cdot \nabla \overline{v} +  2 V(\x) u_N \overline{v} + 2 \beta |u_N|^2 u_N \overline{v}&- 2 i\omega \overline{v} (\bb(\x) \cdot \nabla u_N) \big) \dx\\
&=\lambda_N \int_\Omega u_N \overline{v}  \dx \qquad \forall v \in \X_N,
\end{split}
\end{align}
for some $\lambda_N\in\R$, cf.~\eqref{eq:GPEweak}.

\subsection{Fully discrete $hp$--version GFI}

We now turn to the spatial discretisation of the forward Euler time-stepping gradient flow iteration~\eqref{eq:GFiteration} on the $hp$--finite element subspace $\X_N \subset \H$. To this end, for any $u \in \X_N$, we denote by $\mathsf{G}_N (u) \in \X_N$ the unique solution of 
\begin{equation}
 	a_{u}(\mathsf{G}_N(u),v)=(u,v)_{{\rm L}^2(\Omega)} \qquad \forall v \in \X_N, \label{eq:G_discrete}
\end{equation}
or equivalently,
\[
 	\Re\left(a_{u}(\mathsf{G}_N(u),v)\right)=\Re \left((u,v)_{{\rm L}^2(\Omega)}\right) \qquad \forall v \in \X_N;
\]
cf.~\eqref{eq:G} and~\eqref{eq:Greal}, respectively. For given $u\in\X_N$, note that the computation of $\mathsf{G}_N(u)$ is a standard linear source problem: it can be solved using any appropriate linear solver at the disposal of the user. 
Then, for~$n\ge 0$, the fully discrete GFI is given by
\begin{subequations}\label{eq:discreteGF}
\begin{align} 
 	u_N^{n+1}
 	&=\frac{\widehat{u}_N^{n+1}}{\norm{\widehat{u}_N^{n+1}}_{\L2}}, 
 	\intertext{where}
 	\widehat{u}_N^{n+1}&=  (1-\tau_{N}^n) u_N^n+\frac{\tau_{N}^n}{(\mathsf{G}_N(u_N^n),u_N^n)_{\Ltwo}}\mathsf{G}_N(u_N^n),
\end{align}
\end{subequations}
with a sequence of discrete time steps~$\{\tau_N^n\}_{n \geq 0}$ as in \eqref{eq:timesteps}. 

\subsection{Competitive $hp$-refinements and local energy decays}

In this section we pave the way for an $hp$--adaptive algorithm based on employing
a competitive refinement strategy on each element $\kappa$ in the computational
mesh $\T_N$, $N\ge 0$. The essential idea is to compute the maximal predicted
energy reduction on each element $\kappa\in\T_N$, subject to either a local $h$-- or $p$--refinement.
Elements with the largest maximal predicted decrease in the local energy
functional are then appropriately refined. To this end, we proceed with the following steps.

\begin{figure}
\begin{center}
\includegraphics[scale=0.2]{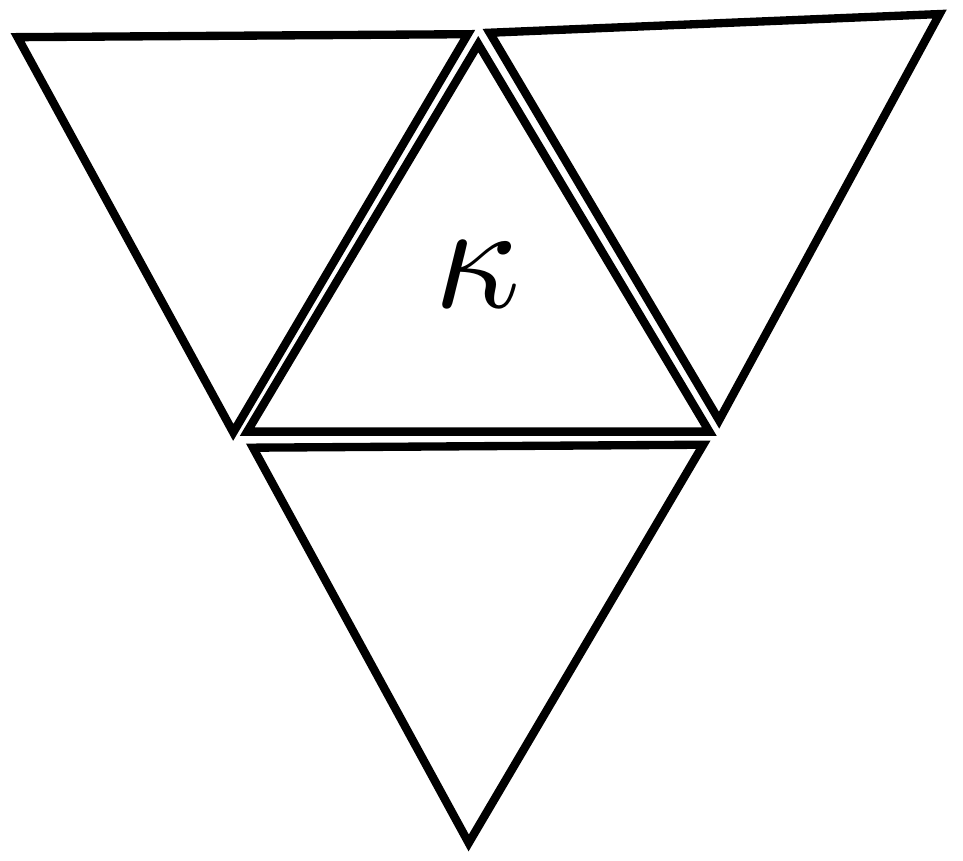} \qquad
\includegraphics[scale=0.2]{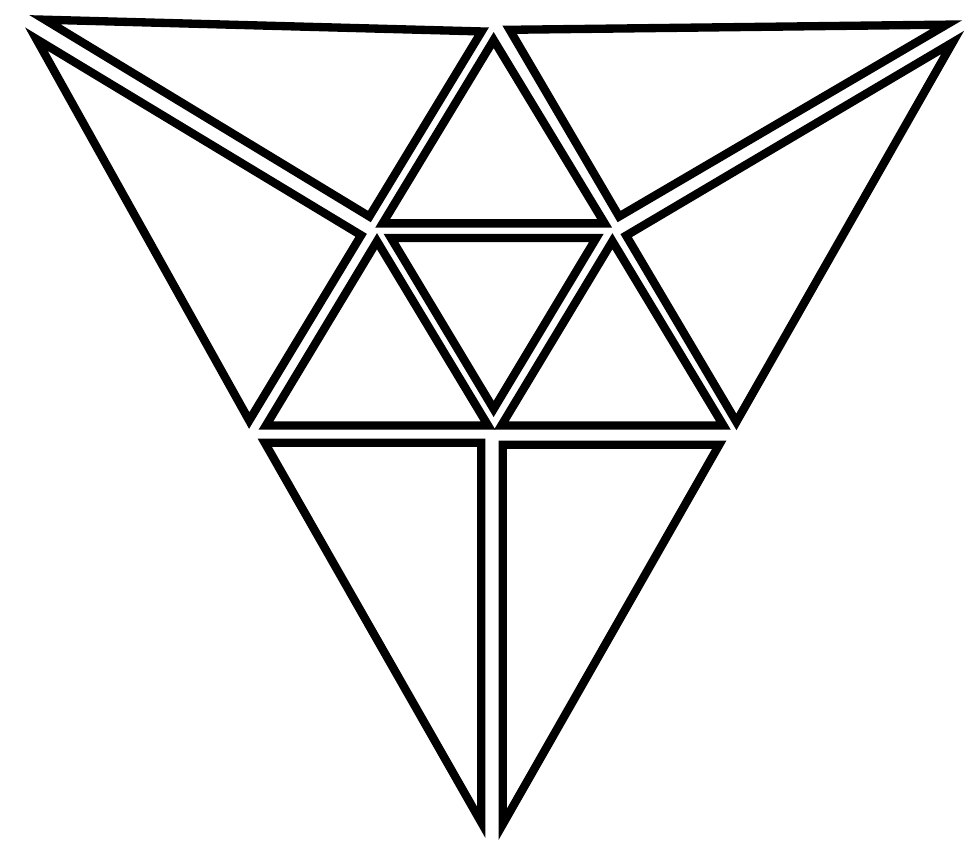} 
\end{center}
\caption{Local element patches in a triangular mesh $\T_N$, $N\ge 0$, in two--dimensions. \emph{Left:} Mesh patch $\TkN$ about an element $\kappa$, which consists of $\kappa$ and its face-wise neighbours. \emph{Right:} Mesh patch $\refmesh$ which is constructed
based on isotropically refining $\kappa$ (red refinement) and on a green refinement of
its neighbours.}
\label{fig:patches}
\end{figure}

\subsubsection{Local patch refinements}
For given $N\ge 0$ and an element $\kappa\in\T_N$, we first construct the
local mesh $\TkN$ comprising of $\kappa$ and its immediate
face-wise neighbours, cf. Fig.~\ref{fig:patches}~(left). Given $\TkN$,
we then uniformly (red) refine element $\kappa$ into~$n_\kappa$ sub-elements, where the introduction of any hanging
nodes may then be removed, if desired, by introducing additional (green) refinements; we denote this locally refined mesh by $\refmesh$, cf. Fig.~\ref{fig:patches}~(right).

\subsubsection{Local $h$-- and $p$--refinements}
We now outline the proposed competitive $h$--/$p$--refinements on a given finite element space $\X_N=\V(\T_N,\p_N$). To this end, for an element $\kappa \in \T_N$, we construct the local mesh patches $\refmesh$ and $\TkN$, as outlined above. Then, local space enrichments in terms of pure $h$--, respectively, $p$--refinement, may be considered as follows:

\begin{enumerate}[($hp$)]
\item[($h$)] Given $\kappa\in \T_N$, based on the locally refined patch $\refmesh$ we construct the finite element space $\Vh^\kappa_{N,\rm h}:=\V(\refmesh,\p^\kappa_{\rm h})$ where we define the global polynomial degree vector by
\[
\left.\p^\kappa_{\rm h}\right|_{\kappa'}:=
\begin{cases}
p_{\kappa}&\text{if }\kappa'\in\refmesh,\\
p_{\kappa'}&\text{otherwise},
\end{cases}
\]
where $p_\kappa$ denotes the polynomial degree on the element $\kappa \in \T_N$.

\item[($p$)] Secondly, we consider the locally refined space $\Vh^\kappa_{N,\rm p}:=\V(\TkN,\p^\kappa_{\rm p})$ for the element $\kappa$ that is solely based on $p$--refinement, where we define a modified (global) polynomial degree vector on $\T_N$ by
\[
\left.\p^\kappa_{\rm p}\right|_{\kappa'}:=
\begin{cases}
p_{\kappa}+1&\text{if }\kappa'\in\T^{\kappa}_N,\\
p_{\kappa'}&\text{otherwise}.
\end{cases}
\]
\end{enumerate}
We emphasise that all functions in $\Vh^\kappa_{N,\rm h}$ and $\Vh^\kappa_{N,\rm p}$ vanish outside of the local patch $\refmesh$, or equivalently, on~$\TkN$, cf.~\eqref{eq:locspace}.

\subsubsection{Local energy decays}
On a given mesh $\T_N$, $N\ge 0$, suppose that we have obtained a sufficiently accurate approximation $u_N^n \in \X_N$ of the fully discrete GPE~\eqref{eq:GPEdiscrete}, after performing $n\ge 1$ steps of the discrete  iteration~\eqref{eq:discreteGF}. We now propose an $hp$-type refinement of the mesh $\T_N$ based on a competitive local reduction strategy of the energy associated to the available approximation $u_N^n$. To this end, for each element $\kappa\in\T_N$, we collect the associated spaces for the proposed local $h$-- and $p$--refinements, cf. above, in the set
\[
\mathcal{V}^\kappa_{N}:=\left\{\Vh^\kappa_{N,\rm h}\right\}
\cup
\left\{\Vh^\kappa_{N,{\rm p}}\right\}.
\]
Recall that both of these spaces contain only functions that are locally supported in a vicinity of the element $\kappa$. Correspondingly, each space $\widehat\V^\kappa_N\in\mathcal{V}^\kappa_{N}$ is spanned by some locally supported and linearly independent functions $\xi^1_\kappa,\ldots,\xi^{m^\kappa}_\kappa$, namely,
\[
\widehat\V^\kappa_N=\Span\{\xi^1_\kappa,\ldots,\xi^{m^\kappa}_\kappa\},
\]
where the number of basis function $m^\kappa$ depends on the underlying refinement. Then, for the proposed local $h$-- or $p$--refinements, and for $u_N^n\in\X_N$ as above, define the augmented space
\begin{equation}\label{eq:W}
\Wh^\kappa_N(u_N^n):=\Span\{\xi^1_\kappa,\ldots,\xi^{m^{\kappa}}_\kappa,u_N^n\},
\end{equation}
which allows us to enrich the approximate solution $u_N^n$ by some local features on $\kappa$ (in terms of the respective $h$-- and $p$--refinements). Specifically, we perform one local fully discrete GFI-step~\eqref{eq:discreteGF} in $\Wh^\kappa_N(u_N^n)\subset \H$ in order to obtain a new local approximation, denoted by $\widetilde u_{N}^{\kappa,n} \in \Wh^\kappa_N(u_N^n)$, with the usual constraint $\norm{\widetilde u_{N}^{\kappa,n}}_{\L2}=1$. 

\begin{remark}
We emphasise that the augmented spaces $\Wh^\kappa_N(u_N^n)$, although they feature \emph{one global} degree of freedom due to the presence of the given approximation $u_N^n$ as a basis function, have small dimensions. Hence, in particular, the fully discrete GFI~\eqref{eq:discreteGF} based on $\Wh^\kappa_N(u_N^n)$ entails hardly any computational cost and can be undertaken in parallel for each of the proposed $h$-- and $p$--refinements (and for each element $\kappa$). 
\end{remark}

In general, we expect the above refinements associated to an element $\kappa\in\T_N$ will lead to a (local) energy decay. Here, we define
\begin{align} \label{eq:localenergydecay_h}
 \Delta \E^\kappa_{N,\rm h}:=(\E(u_N^n) -\E(\widetilde u_{N}^{\kappa,n}))/{\rm dofs}_{\rm h}^\kappa
\end{align}
to be the potential energy reduction, per degree of freedom, when $h$--refinement is employed on $\kappa$, where
${\rm dofs}_{\rm h}^\kappa$ denotes the number of degrees of freedom present in the locally $h$--refined space restricted to $\kappa$.
Similarly, the potential energy reduction, per degree of freedom, when $p$--refinement is utilised is given by
\begin{align} \label{eq:localenergydecay_p}
 \Delta \E^\kappa_{N,{\rm p}}:=(\E(u_N^n) -\E(\widetilde u_{N}^{\kappa,n}))/{\rm dofs}_{\rm p}^\kappa,
\end{align}
where ${\rm dofs}_{\rm p}^{\kappa}$ denotes the number of degrees of freedom present in the locally $p$--refined space restricted to $\kappa$. In this way, for each element $\kappa\in\T_N$, we may compute the value
\begin{equation}
\Delta \E^{\kappa}_{N,\max}
:= \max \left\{ \Delta \E^\kappa_{N,\rm h}, \Delta \E^\kappa_{N,{\rm p}}\right\},
\label{eq:max_error_reduction}
\end{equation}
which indicates the maximal potential energy reduction due to the two types of refinement of the element $\kappa$. 

Finally, we refine the set of elements $\mathcal{K} \in \mathcal{T}_N$ which satisfy the marking criterion
\begin{equation}\label{eq:marking}
\Delta \E^{\kappa}_{N,\max} \geq \theta \max_{\kappa \in \mathcal{T}_N} \Delta \E^{\kappa}_{N,\max},
\end{equation}
where $\theta\in(0,1)$ is a fixed marking parameter. Following \cite{Demkowicz:07,SolinSegethDolezel:04}, cf. also \cite{HoustonWihler:16}, we set $\theta = \nicefrac{1}{3}$. Once the elements in the current mesh $\mathcal{T}_N$ have been $hp$--refined, an $hp$--mesh smoother is then employed in order to construct the enriched finite element space $\X_{N+1}=\V(\T_{N+1},\p_{N+1}$). This is a two-stage process: firstly, hanging nodes created through mesh refinement are removed; here, elements containing two or more hanging nodes are isotropically {\em red} refined. This process is repeated until all elements possess at most one hanging node; the remaining hanging nodes are then removed by introducing temporary (green) refinements to yield a conforming mesh $\mathcal{T}_{N+1}$. Subelements created through this process automatically inherit the polynomial degree of their parent. In the second step, the polynomial degree vector is smoothed by ensuring that the maximal difference in the local polynomial degree between adjacent elements is at most one; this is achieved by enriching the element with the lowest order. The above competitive $hp$--refinement strategy is summarised in Algorithm~\ref{alg:hpFEMbasic}.

\begin{algorithm}[t!]
\begin{algorithmic}[1]
\caption{Energy-based competitive $hp$-adaptive refinement procedure}
\label{alg:hpFEMbasic}
\State{Input a finite element mesh $\T_N$ defined on the domain $\Omega$, an associated polynomial degree vector $\vec p_N$, and a function $u_N\in\V(\T_N,\vec p_N)$.}
\For {each element~$\kappa\in\T_N$}
\myState{\text{/\!/ $h$--refinement of $\kappa$}}
\myState{Find a basis $\{\xi^1_\kappa,\ldots,\xi^{m^{\kappa}}_\kappa\}$ of the locally $h$--enriched space $\Vh^\kappa_{N,\rm h}$.}
\myState{Perform one fully discrete GFI step, cf.~\eqref{eq:discreteGF}, on the local space $\Wh^\kappa_N(u_N)$ from~\eqref{eq:W}, and compute the corresponding predicted energy reduction $\Delta \E^\kappa_{N,\rm h}$, cf. \eqref{eq:localenergydecay_h}.}
\myState{/\!/ $p$--refinement of $\kappa$}
\myState{Find a basis $\{\xi^1_\kappa,\ldots,\xi^{\widetilde{m}^{\kappa}}_\kappa\}$ of the locally $p$--enriched space $\Vh^\kappa_{N,\rm p}$.}
\myState{Perform one fully discrete GFI step, cf.~\eqref{eq:discreteGF}, on the local space $\Wh^\kappa_N(u_N)$ from~\eqref{eq:W}, and compute the corresponding predicted energy reduction $\Delta \E^\kappa_{N,{\rm p}}$, cf. \eqref{eq:localenergydecay_p}.}
\myState{Compute the maximum local predicted error reduction $\Delta \E^{\kappa}_{N,\max}$, cf. \eqref{eq:max_error_reduction}.}
\EndFor
\State{Determine the subset of elements $\mathcal{K}$ which are flagged for refinement, based on the criterion~\eqref{eq:marking}.}
\State{Perform $h$-- or $p$--refinement on each~$\kappa\in\mathcal{K}$ according to which refinement takes the maximum in~\eqref{eq:max_error_reduction} and undertake $hp$-mesh smoothing. This results in a refined global finite element space~$\V(\T_{N+1},\p_{N+1})$.}
\end{algorithmic}
\end{algorithm}

\subsection{$hp$--adaptive strategy}

From a practical viewpoint, once the fully discrete GFI approximation from \eqref{eq:discreteGF} is close to a solution of the discrete GPE~\eqref{eq:GPEdiscrete}, on a given $hp$--finite element space $\X_{N}=\V(\T_{N},\p_{N})$, $N\ge 0$, we expect that any further GFI steps will no longer contribute an essential decay of the underlying energy. In this case, in order to further reduce the energy, we need to enrich the finite element space appropriately. Then, the fully discrete GFI iteration is reinitiated on the new space, $ \X_{N+1}=\V(\T_{N+1},\p_{N+1})$, and so on.
More specifically, for~$N \ge 0$, suppose that we have performed a reasonable number~$n \ge 1$ (possibly depending on~$N$) of GFI-iterations~\eqref{eq:discreteGF} in $\X_{N}$. Consider now an $hp$--enriched space~$\X_{N+1}=\V(\T_{N+1},\p_{N+1})$ as obtained by Algorithm~\ref{alg:hpFEMbasic}. Then we may embed (or project, if refinements are not hierarchical) the final guess $u^{n}_{N}\in \X_{N}$ on the previous space into the enriched finite element space~$\X_{N+1}$ in order to obtain an initial guess on 
\begin{equation}\label{eq:u0}
\X_N\ni u^n_{N}\hookrightarrow u^{0}_{N+1}\in\X_{N+1}.
\end{equation} 
For each GFI-iteration we monitor two quantities. Firstly, we introduce the energy \emph{increment} on each iteration given by
\begin{align}\label{eq:inc} 
	\incNk := \E(u_N^{n-1}) -  \E(u_N^n),\qquad n\ge 1.
\end{align}
Secondly, we compare $\incNk$ to the total energy \emph{decay} on the current space $\X_N$, i.e.,
\begin{align} \label{eq:diff}
	\dENk := \E(u_N^{0}) -  \E(u_N^n),\qquad n\ge1.
\end{align}
We stop the iteration for~$n\ge 1$ as soon as $\incNk$ becomes small compared to $\dENk$, i.e., once there is no notable (relative) benefit in performing any further discrete GFI steps on the current space $\X_{N}$. Specifically, for~$n\ge 1$, this is expressed by the bound
\[
	\incNk \le \gamma\dENk,
\]
for some parameter $0<\gamma<1$. 
We implement this procedure in Algorithm~\ref{alg:adaptive}.

\begin{algorithm}[t!]
\caption{$hp$--Adaptive fully discrete finite element gradient flow iteration}
\label{alg:adaptive}
\begin{algorithmic}[1]
\State Prescribe three parameters~$\theta=\nicefrac{1}{3}$, $\gamma\in(0,1)$, and $0<\tol\ll 1$.
\State Choose a sufficiently fine initial mesh~$\T_0$, and an initial guess~$u^0_0\in \S_\H$ with $(u^0_0,1)_{\L2} \geq 0$, cf.~\eqref{eq:unorm}. Set~$N:=0$.
\Loop 
	\State Set $n:=1$, perform one fully discrete GFI-step~\eqref{eq:discreteGF} on $\X_N$ to obtain $u_{N}^{1}\in\X_N$.
	\State Compute the indicator $\mathrm{inc}_N^1$ (which equals~$\diff_N^1$) from~\eqref{eq:inc}. 
	\While {$\incNk > \gamma\,\dENk$} 
		\State Update $n\gets n+1$.
		\State Perform one GFI-step~\eqref{eq:discreteGF} in $\X_N$ to obtain $u_{N}^{n}$ (from $u_N^{n-1}$).
		\State Compute the indicators $\incNk$ and $\dENk$ from~\eqref{eq:inc} and~\eqref{eq:diff}, respectively.
	\EndWhile
	\If {$\dENk > \tol\, \E(u_N^n)$}
		\State \multiline{\textsc{Mark} and adaptively $hp$--\textsc{refine} the mesh~$\T_N$ using Algorithm~\ref{alg:hpFEMbasic} in order to generate a new mesh~$\T_{N+1}$.}
		\State Define $u_{N+1}^0\in\X_{N+1}$ from $u_N^n\in \X_{N}$ by canonical embedding (or by projection), cf.~\eqref{eq:u0}.
		\State Update~$N\gets N+1$.
	\Else 
	\State \Return $u^n_N$.
	\EndIf
\EndLoop
\end{algorithmic}
\end{algorithm}

\begin{remark}
In practical computations, in order to guarantee a positive energy decay in each iteration step, we propose the time step strategy within~\eqref{eq:discreteGF} given by
\begin{align*} 
 \tau_N^{n} = \max\left\{2^{-m}:\, \E(u_N^{n+1}(2^{-m})) < \E(u_N^n),\,m\ge 0\right\},\qquad  n \geq 0.
\end{align*}
where, for~$0<s\le 1$, we write $u_N^{n+1}(s)$ to denote the output of the discrete GFI \eqref{eq:discreteGF} based on the time step $\tau_N^n=s$ and on the previous approximation~$u_N^n$. We have observed in several examples that for the choice $\tau_N^0=1$, i.e., using $m=0$ above, no time correction was needed. For that reason, and for the sake of keeping the computational cost minimal, we fix the time step $\tau=1$ in the local one-step GFI in Algorithm~\ref{alg:hpFEMbasic}; we use, however, the time step strategy for the global GFI in Algorithm~\ref{alg:adaptive}.
\end{remark}

\section{Numerical Experiments} \label{sec:numerics}

In this section, we present a series of numerical experiments to investigate the practical performance of the proposed $hp$--refinement algorithm outlined in Algorithms~\ref{alg:hpFEMbasic}~\&~\ref{alg:adaptive}. Throughout this section, we set $\theta=\nicefrac{1}{3}$, cf. above, and $\gamma=10^{-3}$. The selection of $\tol$ is problem dependent in order to compute a highly accurate value of the ground-state energy for the example at hand. Throughout this section we compare the performance of the proposed $hp$--adaptive refinement strategy with the corresponding algorithm based on exploiting only local mesh subdivision, i.e., $h$--refinement, cf.~\cite{HeidStammWihler:21}. Furthermore, to compute the solution of the underlying linear problem~\eqref{eq:G_discrete}, we employ the direct MUltifrontal Massively Parallel Solver (MUMPS), see \cite{MUMPS:1,MUMPS:2,MUMPS:3}.

\subsection{GPE without angular momentum operator $\omega = 0$}

Firstly, we consider a series of numerical examples where the angular momentum operator is 
omitted, i.e., when  $\omega = 0$. In particular, we demonstrate the performance of the proposed $hp$--adaptive refinement algorithm on a series of examples presented in~\cite{HeidStammWihler:21}. Throughout this section, we select the initial ground-state $u_0^0 \in \X_0 \subset \H$, such that $u_0^0(\x)=c$ for any node $\x$ in the interior of the initial mesh $\mathcal{T}_0$, where, $c > 0$ is the appropriate real constant which fulfils the norm constraint $\norm{u_0^0}_{\L2}=1$.

\subsubsection{Non-linear GPE with harmonic confinement potential}

In this first example we consider a non-linear Bose-Einstein condensate with $\beta = 1000$, $V = \nicefrac{1}{2}(x^2+y^2)$, ($\omega = 0$), on the computational domain $\Omega=(-6,6)^2$. This example has been previously considered in \cite{henning2020sobolev} and \cite{HeidStammWihler:21}, where the approximations $\E(u_\mathrm{GS}) \approx 11.9860647$ and $\E(u_\mathrm{GS}) \approx 11.98605121$ have been computed, respectively, for the energy of the ground-state. Based on employing the $hp$--refinement algorithm proposed in this article, cf.~Algorithms~\ref{alg:hpFEMbasic}~\&~\ref{alg:adaptive}, we compute a reference value of $\E(u_\mathrm{GS}) \approx \underline{11.98605114667}1144$. Here, the underlined digits are stable in the sense that our computations indicate that they do not change as further iterations of the $hp$--refinement strategy are computed.

\begin{figure}[t!]
\begin{center}
\begin{tabular}{cc}
\includegraphics[width=0.45\textwidth]{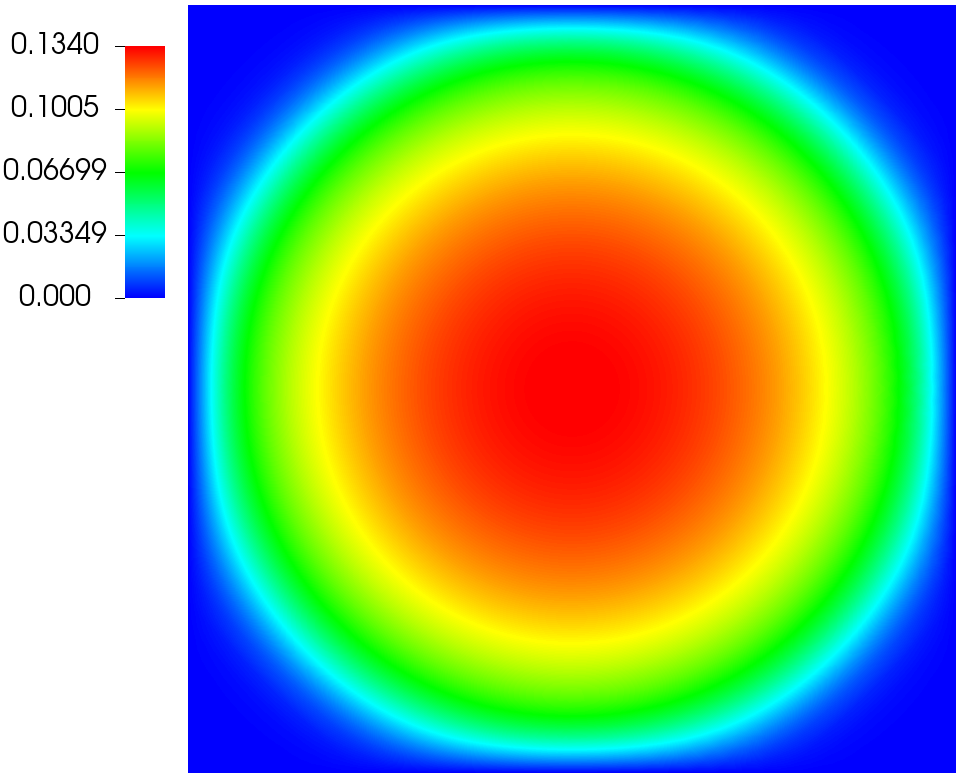} &
\includegraphics[width=0.45\textwidth]{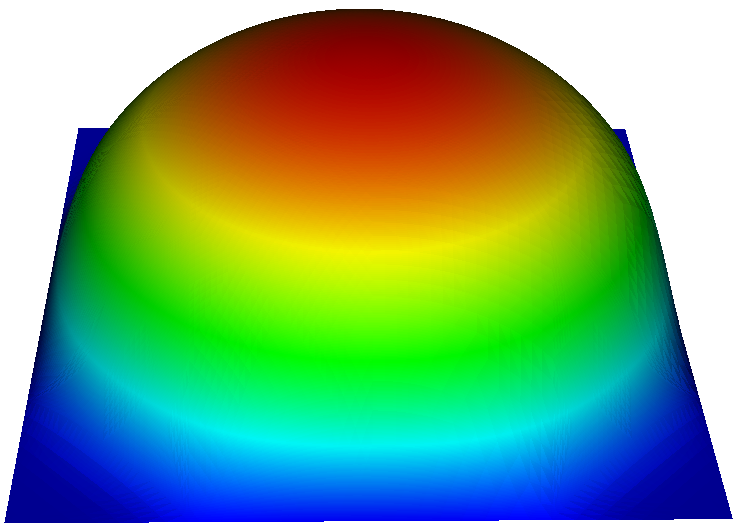} \\
\multicolumn{2}{c}{(a)} \\
& \\
\includegraphics[scale=0.43]{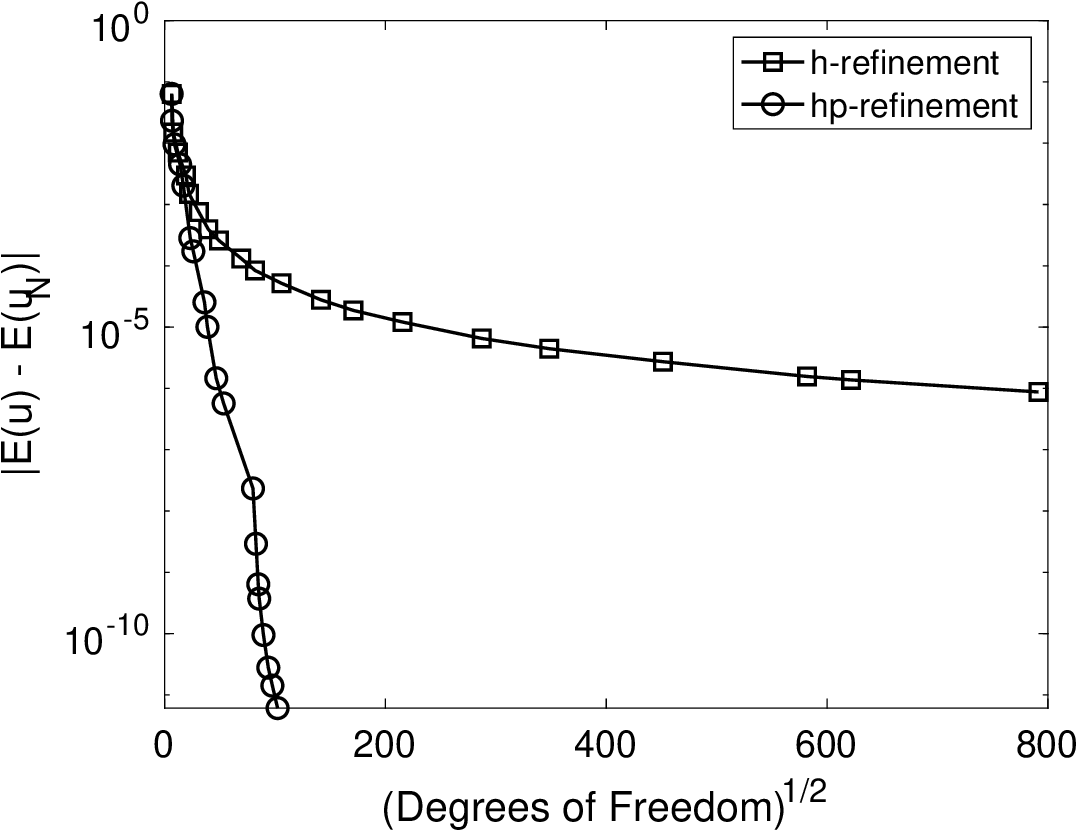} &
\includegraphics[scale=0.43]{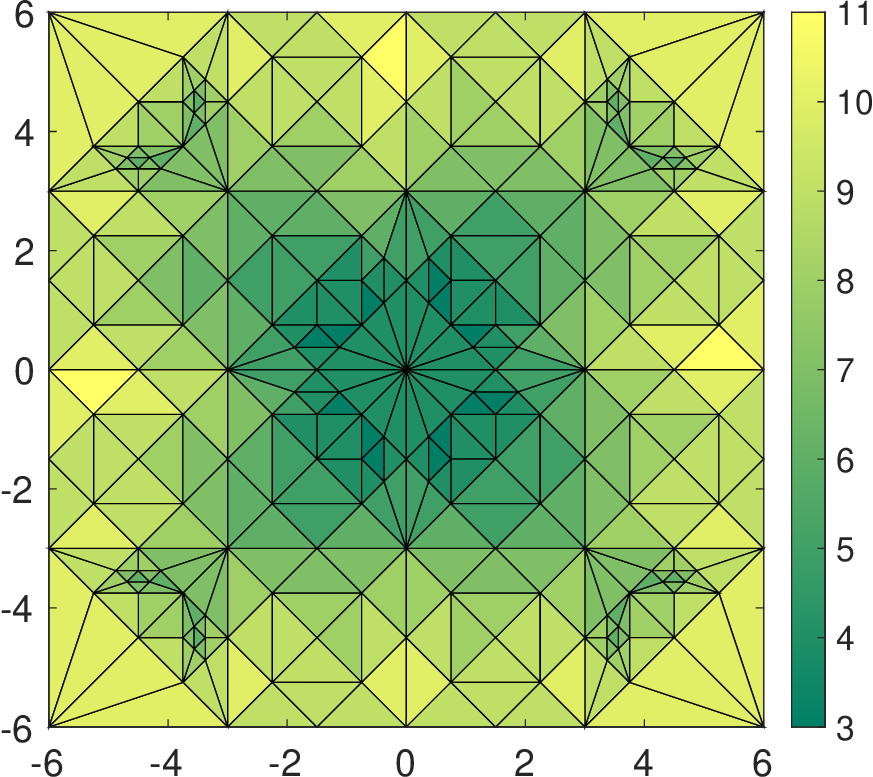} \\
(b) & (c)
\end{tabular}
\end{center}
\caption{Non-linear GPE with harmonic confinement potential. (a) Computed ground-state $u_N$; (b) Comparison of the error with respect to the square root of the
number of degrees of freedom; (c) Final $hp$-mesh.}
\label{fig:harmonic_confinement}
\end{figure}

The computed approximation to the ground-state $u^n_N \equiv u_N\approx u_\mathrm{GS}$ is depicted in Figure~\ref{fig:harmonic_confinement}(a), the shape of which resembles an `upside down bowl'. Furthermore, in Figure~\ref{fig:harmonic_confinement}(b) we illustrate the performance of the proposed $hp$--adaptive algorithm, cf. Algorithms~\ref{alg:hpFEMbasic}~\&~\ref{alg:adaptive}. Here, we plot the error of the energy $\E(u_N)-\E(u_{\mathrm{GS}})$ against the square root of the number of degrees of freedom employed within the finite element space $\X_{N}$, on a linear–log scale, based on employing both $h$-- and $hp$--refinement. We observe that the $hp$--refinement algorithm leads to an exponential decay in the error of the approximated ground-state energy as the finite element space $\X_{N}$ is adaptively enriched: indeed, on a linear-log plot, the convergence line is roughly straight. Moreover, we observe the superiority of $hp$--refinement in comparison with a standard $h$--refinement algorithm, cf.~\cite{HeidStammWihler:21}, in the sense that the former refinement strategy leads to several orders of magnitude reduction in the error in $\E$, for a given number of degrees of freedom, than the corresponding quantity computed exploiting mesh subdivision only.

Finally, in Figure~\ref{fig:harmonic_confinement}(c) we show the final $hp$--refined mesh generated by the proposed refinement strategy. Here we observe that some $h$--refinement has been undertaken around the base of the (upside down) bowl profile of the ground-state, as well as in the centre of the domain. The polynomial degree has then been enriched where the ground-state has steep gradients, as we would expect.

\subsubsection{Non-linear GPE with optical lattice potential} 

\begin{figure}[t!]
\begin{center}
\begin{tabular}{cc}
\includegraphics[width=0.42\textwidth]{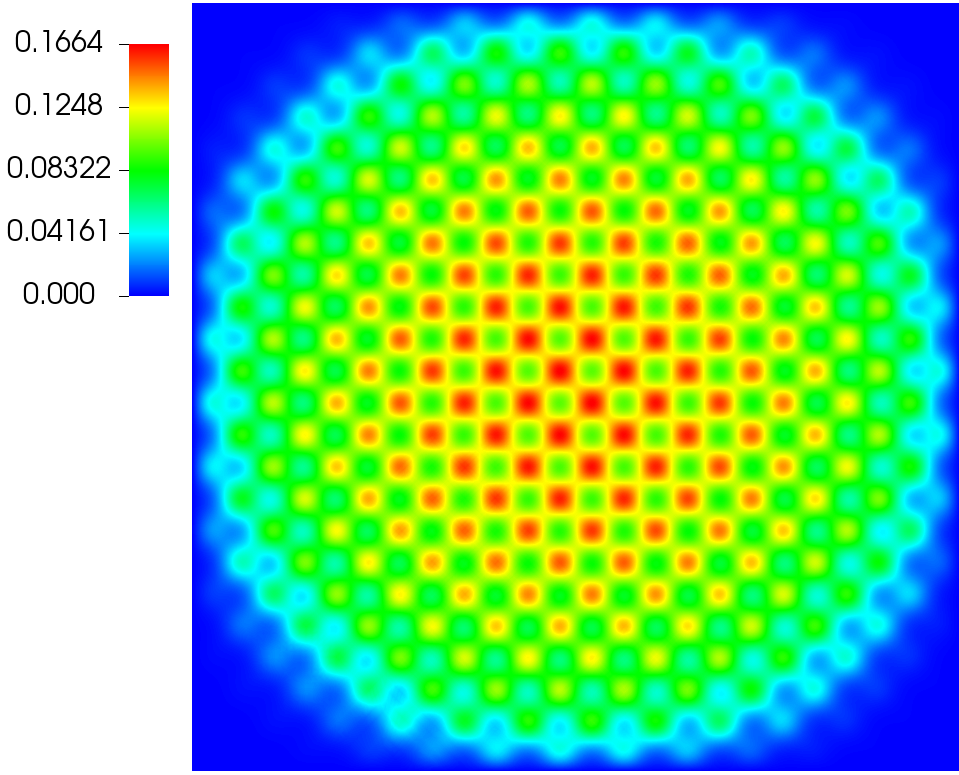} &
\includegraphics[width=0.35\textwidth]{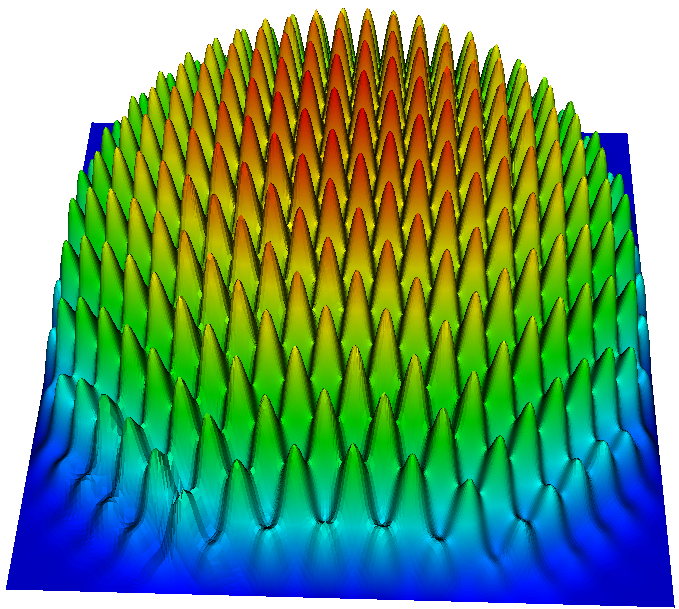} \\
\multicolumn{2}{c}{(a)} \\
& \\
\includegraphics[scale=0.4]{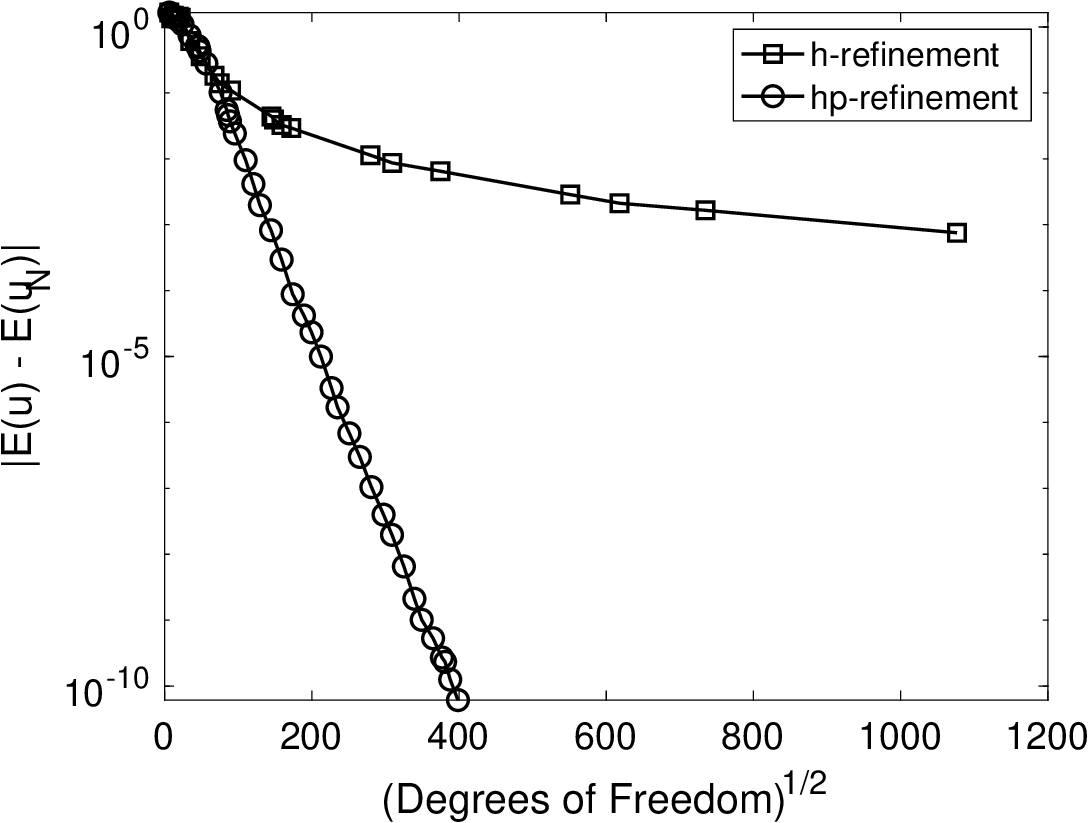} &
\includegraphics[scale=0.4]{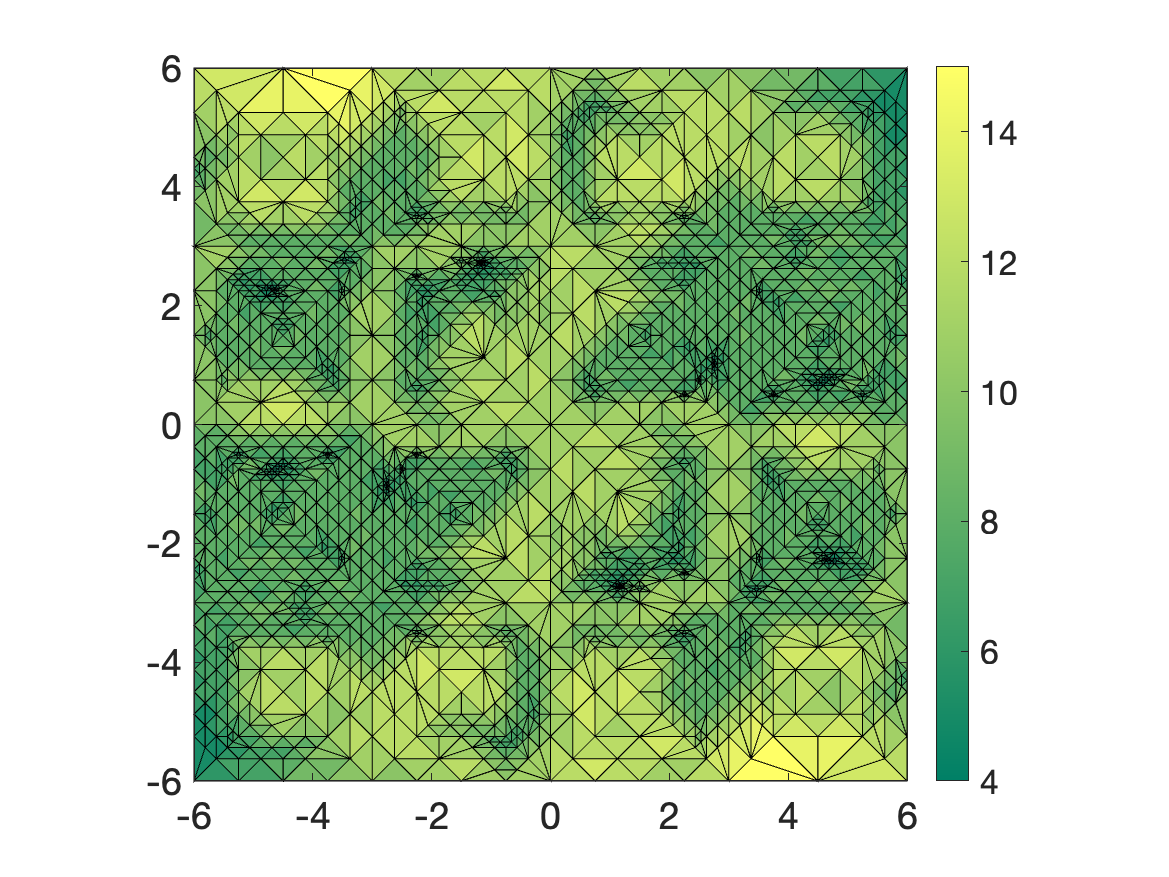} \\
(b) & (c)
\end{tabular}
\end{center}
\caption{Non-linear GPE with optical lattice potential. (a) Computed ground-state $u_N$; (b) Comparison of the error with respect to the square root of the
number of degrees of freedom; (c) Final $hp$-mesh.}
\label{fig:optical_lattice_potential}
\end{figure}

In this second example, we consider a more challenging problem with an oscillating potential function $V$; to this end, we select $\Omega = (-6,6)^2$, $\beta=1000$,
$V = \nicefrac{1}{2}(x^2+y^2)+20+20 \sin(2\pi x) \sin(2\pi y)$, and again $\omega=0$. The computed approximation to the ground-state is depicted in Figure~\ref{fig:optical_lattice_potential}(a); in this setting $u_\mathrm{GS}$ resembles a `hedgehog' structure consisting of a large number of very thin spikes.
Computations undertaken in \cite{henning2020sobolev} and \cite{HeidStammWihler:21} give values of the ground-state energy $\E(u_\mathrm{GS}) \approx 30.40965$ and $\E(u_\mathrm{GS}) \approx 30.387533$, respectively. Based on exploiting $hp$--adaptivity, we have computed a reference value of $\E(u_\mathrm{GS}) \approx \underline{30.3874145736}32611$.

The performance of the proposed $hp$--refinement algorithm is depicted in Figure~\ref{fig:optical_lattice_potential}(b); here, again we observe exponential convergence of the error in the computed ground-state energy, despite the fine scale structures present in the ground-state. Moreover, $hp$--refinement clearly outperforms its $h$--version counterpart. The final $hp$--mesh shown in Figure~\ref{fig:optical_lattice_potential}(c) indicates that some $h$--refinement is necessary across most of the domain, with again $p$--enrichment being undertaken in smooth regions.

\subsubsection{Non-linear energy functional with a nonsymmetric potential
$V$}

The final example before we consider the addition of angular momentum involves employing the nonsymmetric potential $V = \nicefrac{1}{2}\left(x^2+y^2 + 8 \exp(-(x-1)^2-y^2)\right)$, with $\beta = 200$ and $\omega = 0$ posed on the domain $\Omega=(-8,8)^2$, cf. \cite{BaoDu:04} and \cite{HeidStammWihler:21} who computed the approximations
$\E(u_\mathrm{GS}) \approx 5.8507$ and $\E(u_\mathrm{GS}) \approx 5.85058738$, respectively. Here, we have computed
$\E(u_\mathrm{GS}) \approx \underline{5.850587113515}1475$. The ground-state consists of a single hill with a small circular region removed, cf. Figure~\ref{fig:nonsymmetric_potential}(a).

Figure~\ref{fig:nonsymmetric_potential}(b) again highlights the potential benefits of employing $hp$--adaptivity: exponential convergence of the error in the computed ground-state energy is observed, as well as its superiority in terms of the computed error versus the number of degrees of freedom employed, when compared with mesh subdivision ($h$--refinement). The final $hp$--mesh depicted in Figure~\ref{fig:nonsymmetric_potential}(c) highlights that some $h$-refinement is necessary around the top of the hill, while largely $p$--enrichment is undertaken elsewhere.

\begin{figure}[t!]
\begin{center}
\begin{tabular}{cc}
\includegraphics[width=0.45\textwidth]{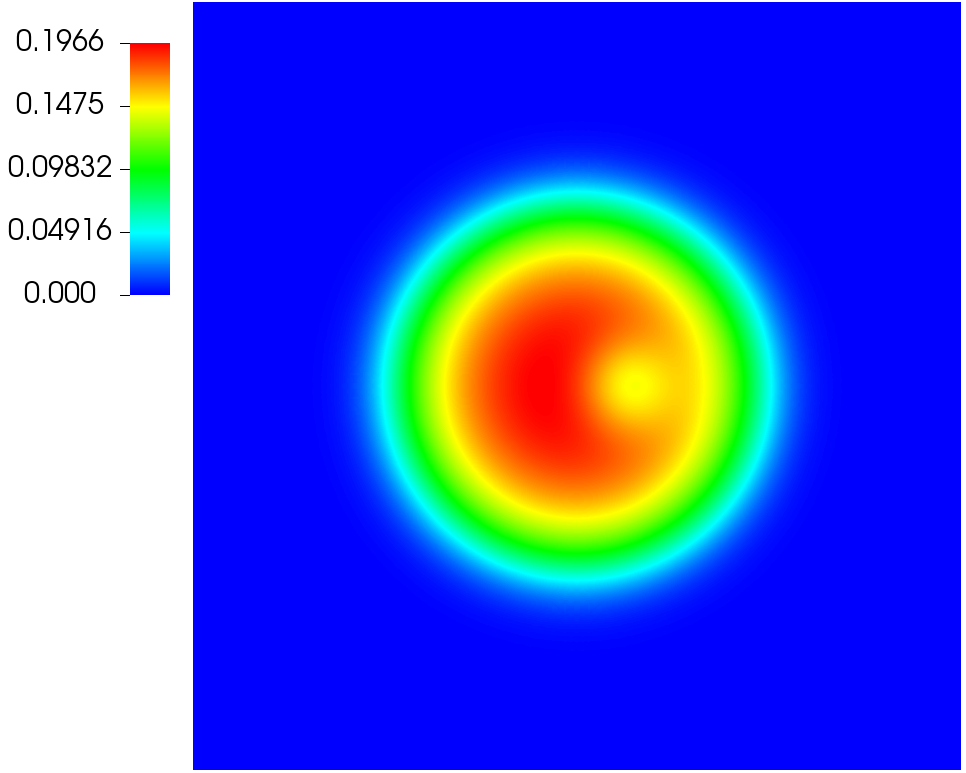} &
\includegraphics[width=0.4\textwidth]{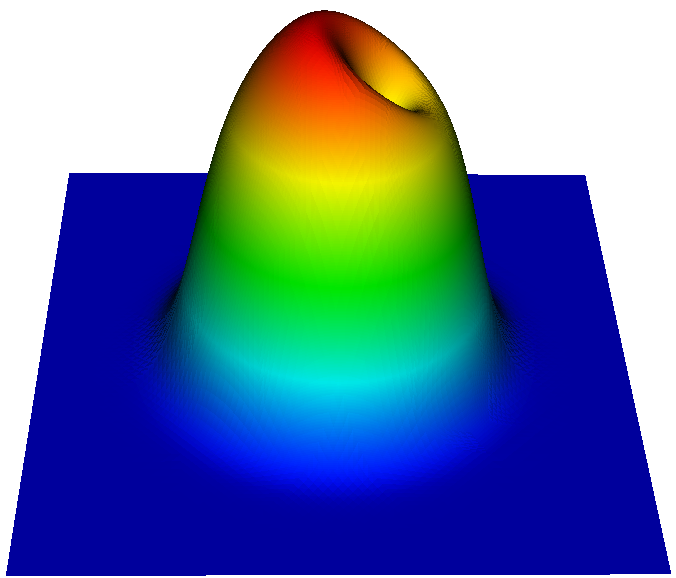} \\
\multicolumn{2}{c}{(a)} \\
& \\
\includegraphics[scale=0.43]{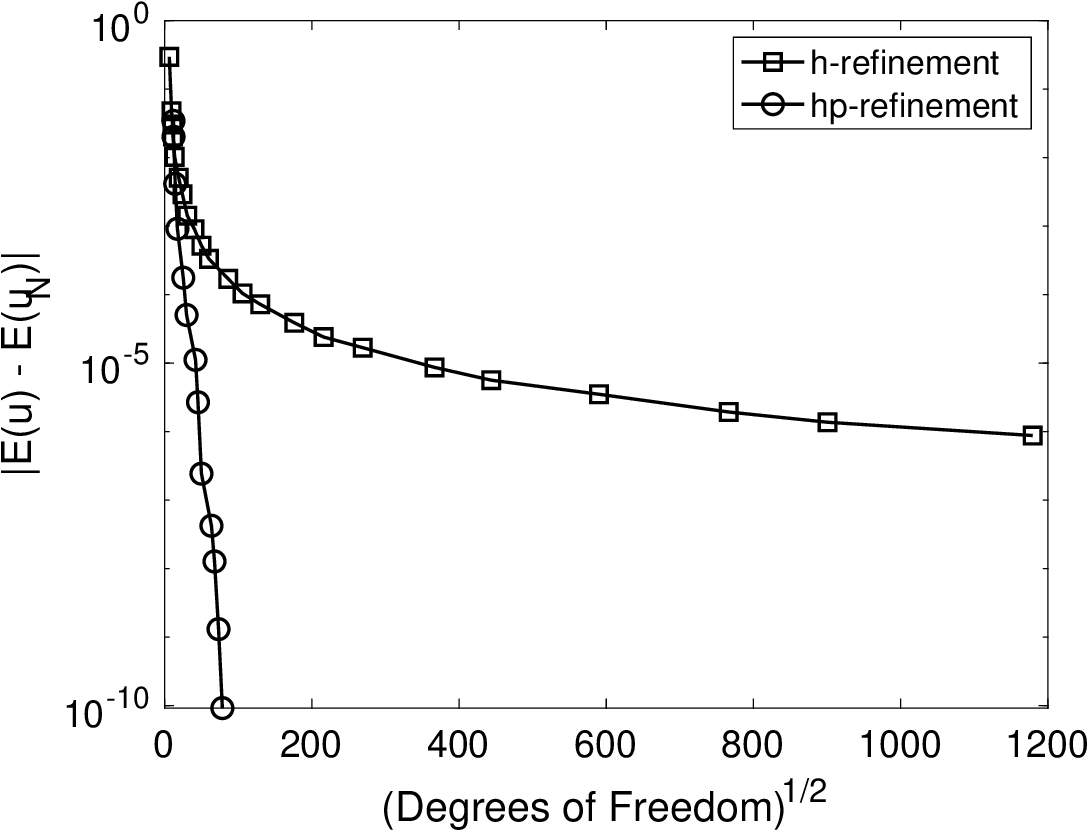} &
\includegraphics[scale=0.43]{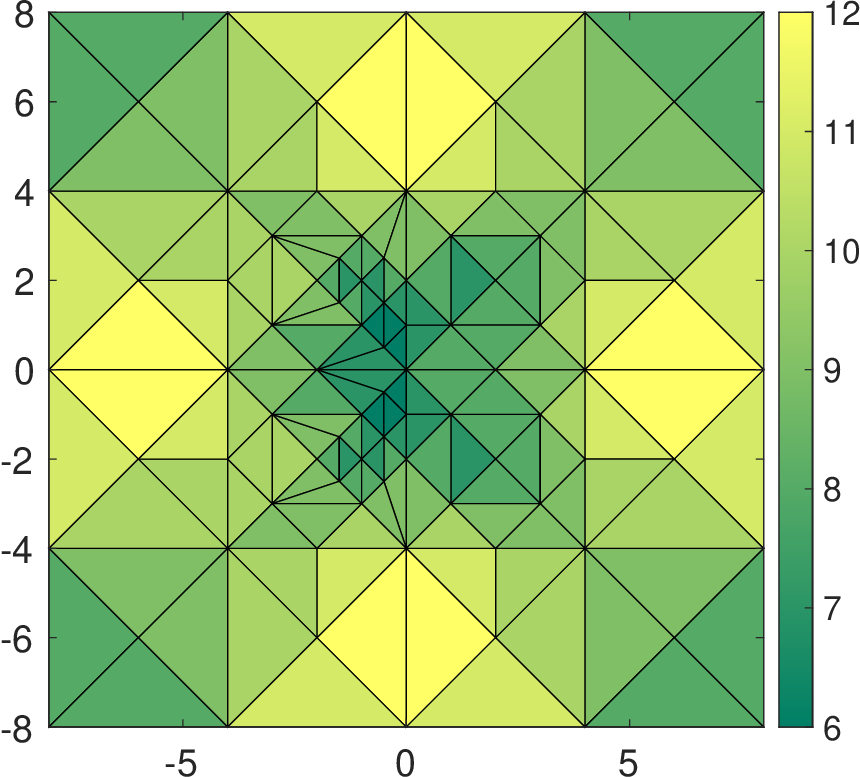} \\
(b) & (c)
\end{tabular}
\end{center}
\caption{Non-linear GPE with a nonsymmetric potential $V$. (a) Computed ground-state $u_N$; (b) Comparison of the error with respect to the square root of the
number of degrees of freedom; (c) Final $hp$-mesh.}
\label{fig:nonsymmetric_potential}
\end{figure}

\subsection{GPE with angular momentum operator $\omega>0$}

\begin{table}[t!]
    \begin{center}
    \begin{tabular}{c|c|c|c}
       $\omega$  & $\beta$ & $L$  & $\E(u_\mathrm{GS})$  \\ \hline
       $0.5$     &  $10$   & $6$  & $\underline{1.592319024681}3326$ \\
       $0.75$    &  $100$  & $6$  & $\underline{3.38102774209}47901$  \\
       $0.8$     &  $500$  & $10$ & $\underline{6.09974}39947822603$ \\
       $0.9$     & $1000$  & $12$ & $\underline{6.360}9757543503642$ \\
       \\
    \end{tabular}
    \end{center}
    \caption{Energy functional values for $\omega>0$.}
    \label{table:energies_for_positive_omega}
\end{table}

\begin{figure}[t!]
\begin{center}
\begin{tabular}{cc}
\includegraphics[width=0.45\textwidth]{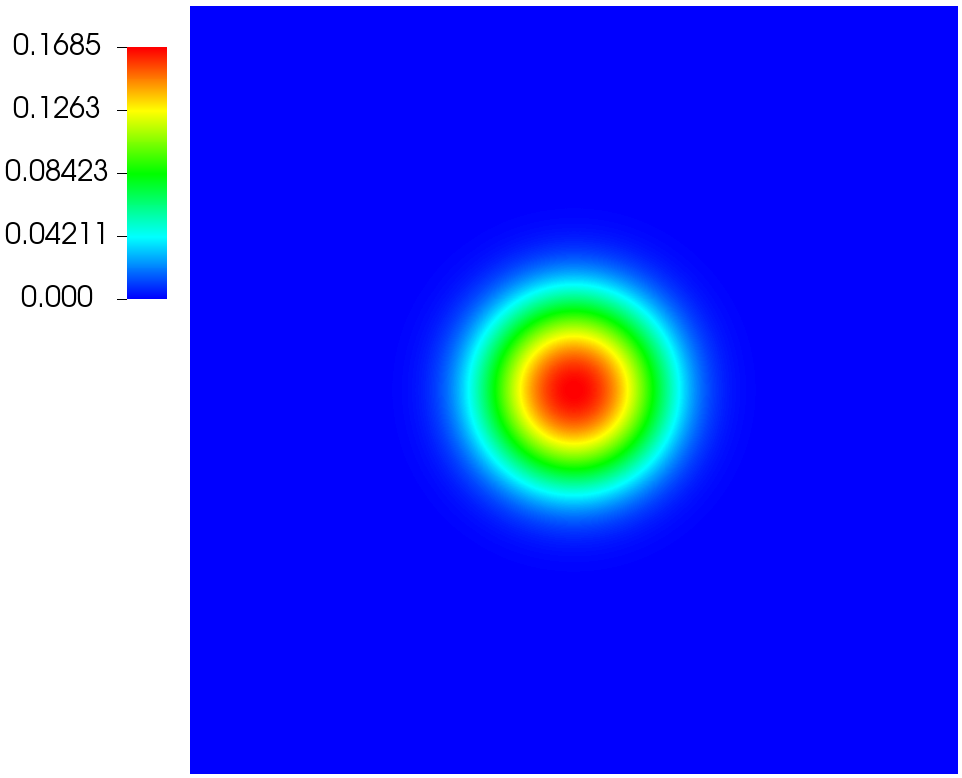} &
\includegraphics[width=0.35\textwidth]{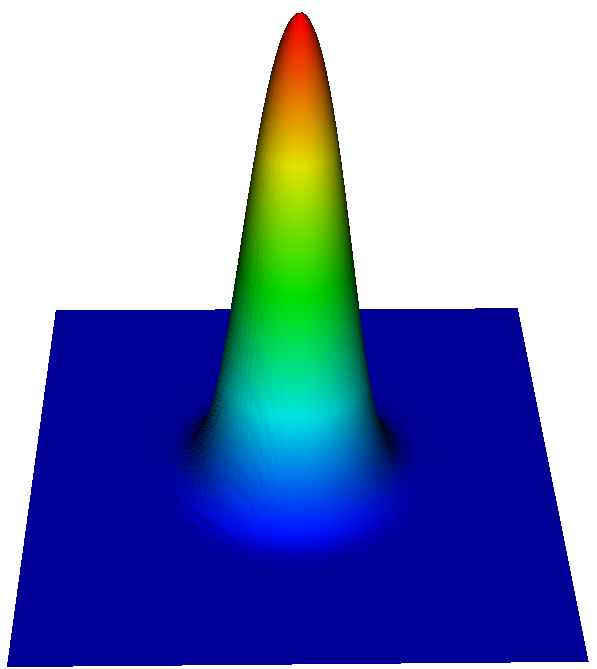} \\
\multicolumn{2}{c}{(a)} \\
& \\
\includegraphics[scale=0.43]{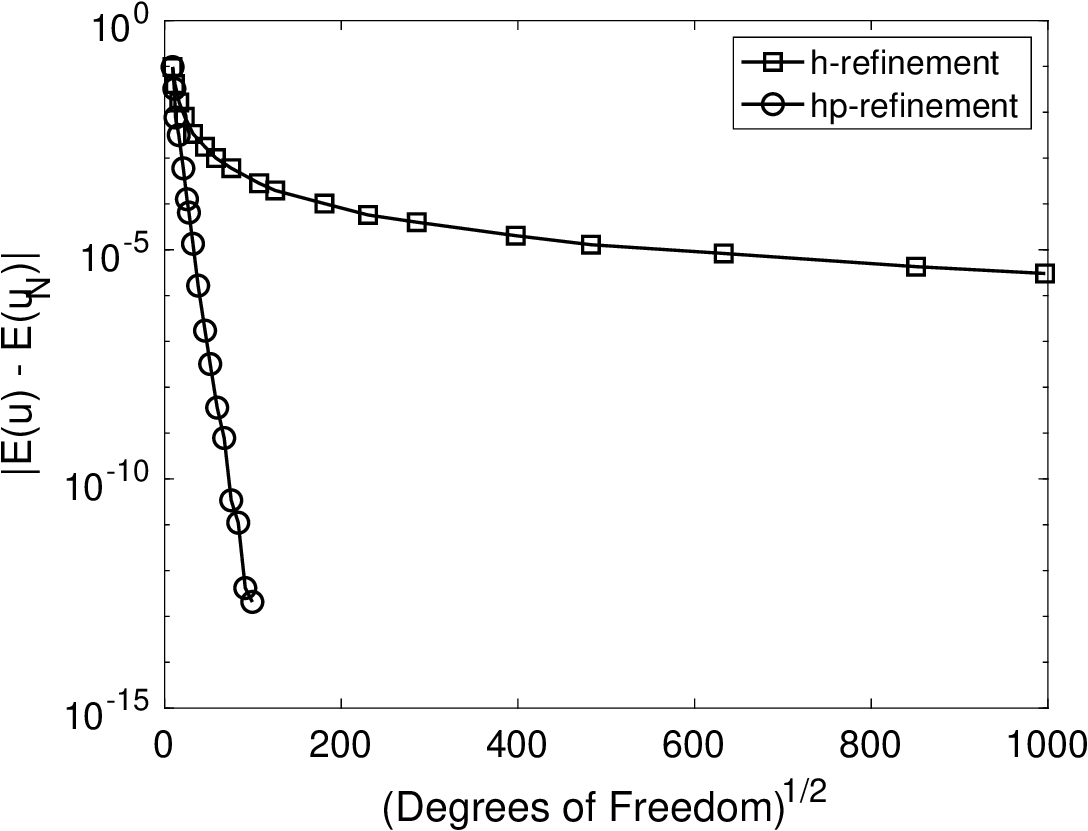} &
\includegraphics[scale=0.43]{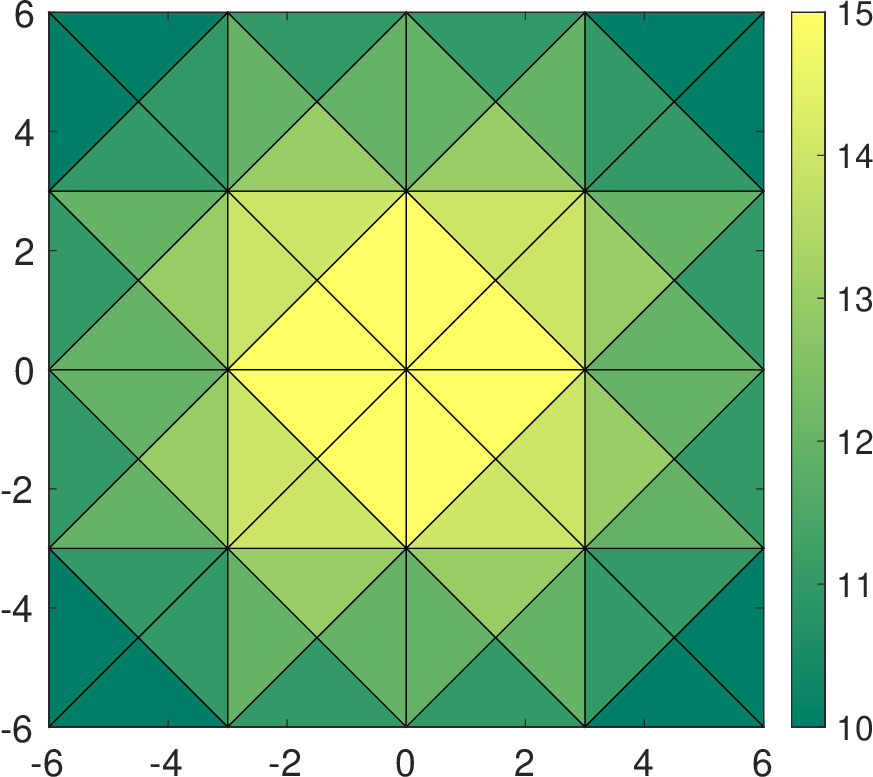} \\
(b) & (c)
\end{tabular}
\end{center}
\caption{GPE with angular momentum: $\omega = 0.5$, $\beta = 10$. (a) Computed ground-state density $|u_N|^2$; (b) Comparison of the error with respect to the square root of the
number of degrees of freedom; (c) Final $hp$-mesh.}
\label{fig:omega_0.5_beta_10}
\end{figure}

Finally, in this section, we now turn our attention to the case when the angular momentum operator is present in the underlying energy functional, i.e., when $\omega>0$. To this end, we set $\Omega = (-L,L)$, for some $L>0$, $V = \nicefrac{1}{2}(x^2+y^2)$, and consider the range of $\omega$ and $\beta$, for a given $L$, presented in Table~\ref{table:energies_for_positive_omega}. In addition, here we report the minimal energy computed using the proposed $hp$--refinement algorithm based on selecting the initial condition
\begin{align*}
u^0(x,y)=\frac{(1-\omega)\varphi_0(x,y)+\omega (x+iy) \varphi_0(x,y)}{\norm{(1-\omega)\varphi_0(x,y)+\omega (x+iy) \varphi_0(x,y)}_{\L2}},
\end{align*}
where
\[
\varphi_0(x,y)=\frac{1}{\sqrt{\pi}} \exp\left(-\nicefrac{(x^2+y^2)}{2}\right),
\]
with zero boundary conditions, cf. \cite{bao2005ground,danaila2017computation,Wu:2017}. We remark that the dependence on the computed ground-state on the choice of $u^0$ has been studied, for example, in \cite{bao2005ground,Wu:2017}; both articles recommend the selection of $u^0$ above (or its complex conjugate), particularly for large $\omega$.
Various values for the ground-state energy have been computed in the literature; in particular, \cite{bao2005ground} computed $\E(u_\mathrm{GS}) \approx 1.5914$ and $\E(u_\mathrm{GS}) \approx 3.371$ for the cases $\omega = 0.5$, $\beta=10$ and $\omega = 0.75$, $\beta=100$, respectively. From \cite{Wu:2017}, for $\omega = 0.8$ and $\beta=500$ values of $\E(u_\mathrm{GS})$ in the range $[6.0997,6.1055]$ were computed. Finally, when $\omega = 0.9$, and $\beta=1000$, a range of values $\E(u_\mathrm{GS}) \in [6.3615,6.3621]$ were computed in \cite{danaila2010new}, and in \cite{Wu:2017} $\E(u_\mathrm{GS})$ was estimated to be in the range $[6.3603,6.3608]$.
We point out that the different ground-state energies have been computed using a variety of numerical methods, employing different initial conditions and at times different computational domains $\Omega$.

\begin{figure}[t!]
\begin{center}
\begin{tabular}{cc}
\includegraphics[width=0.45\textwidth]{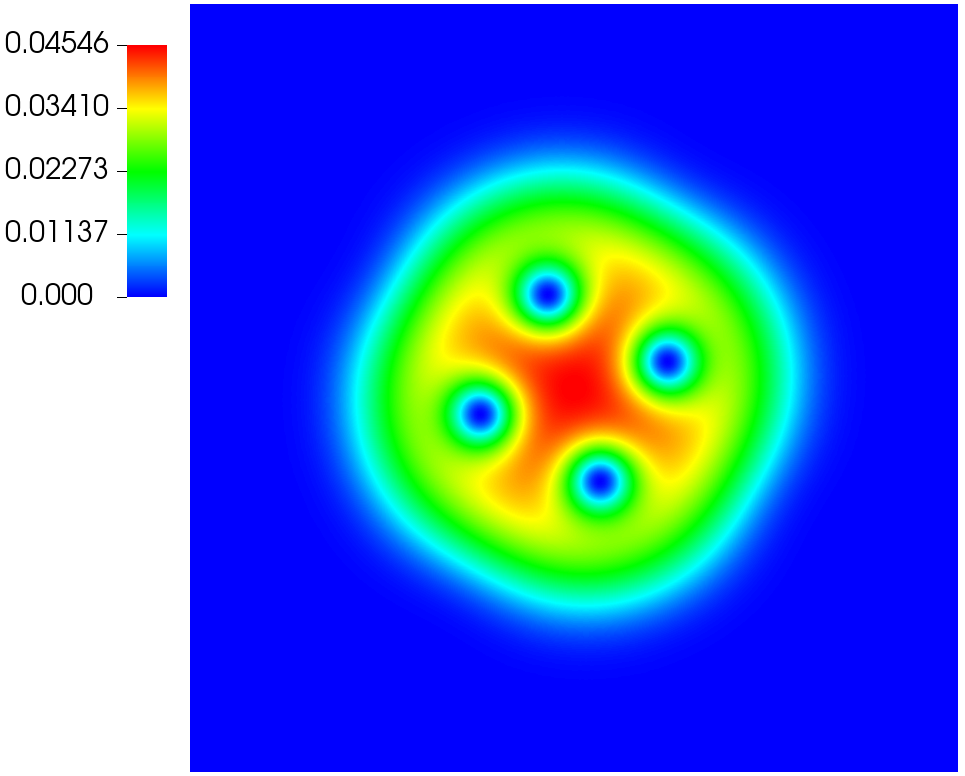} &
\includegraphics[width=0.4\textwidth]{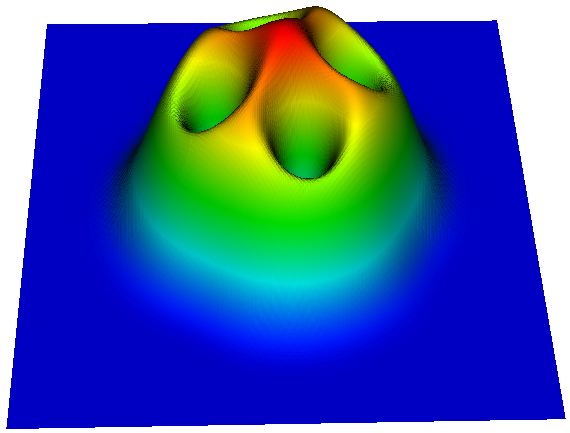} \\
\multicolumn{2}{c}{(a)} \\
& \\
\includegraphics[scale=0.4]{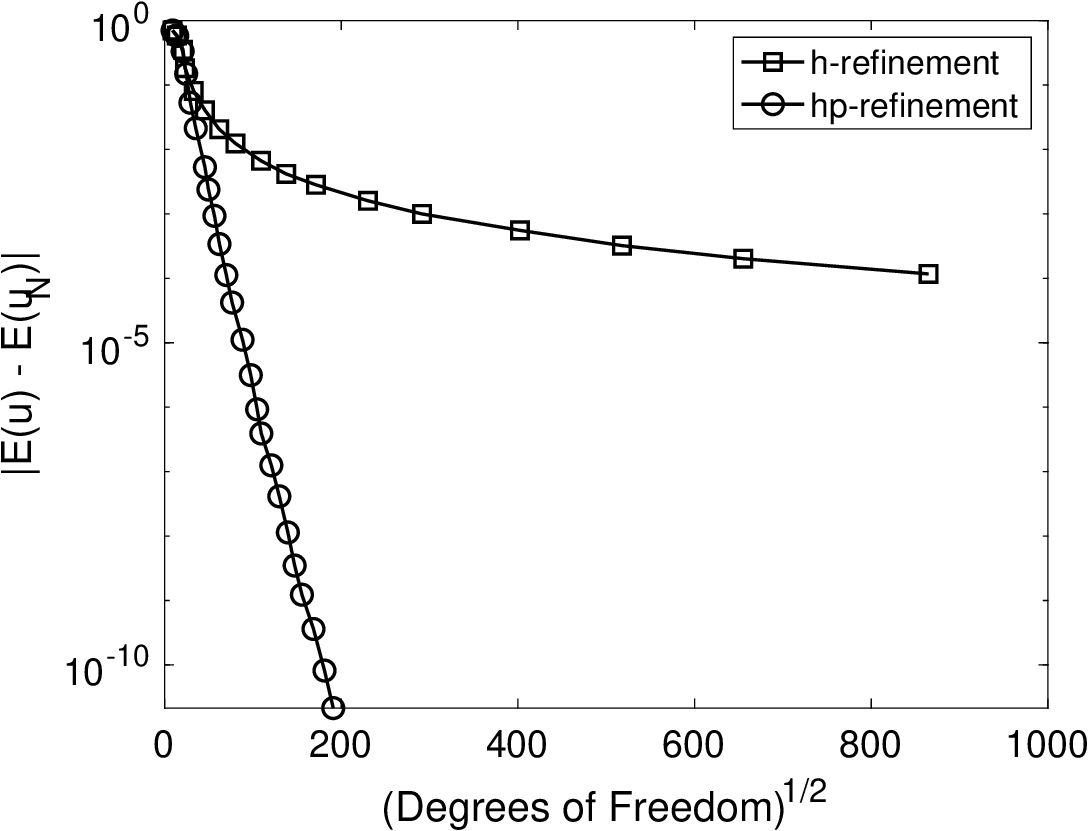} &
\includegraphics[scale=0.43]{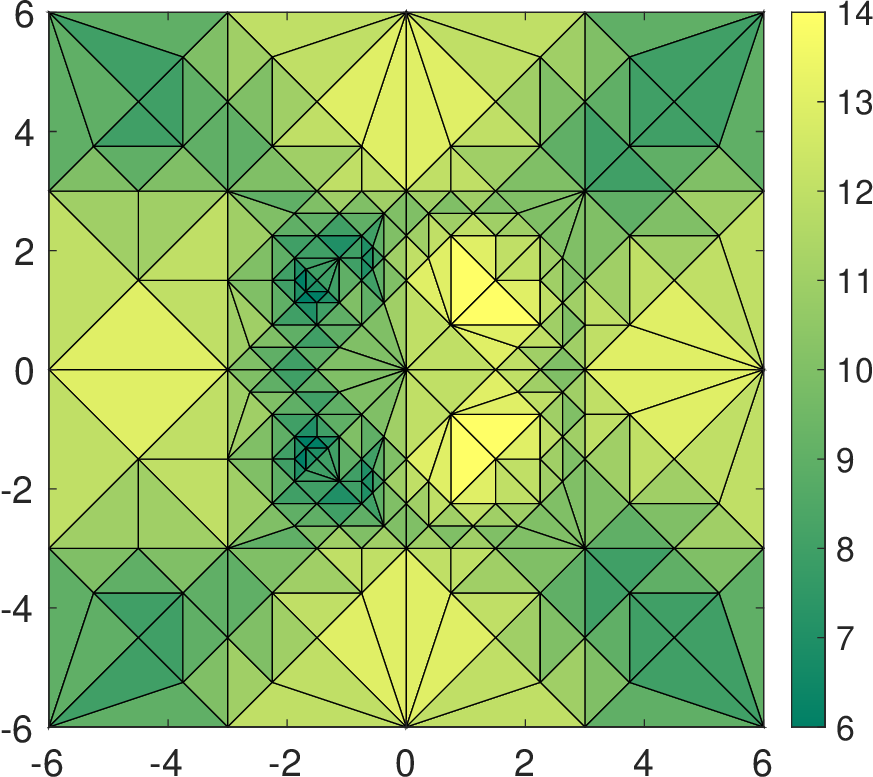} \\
(b) & (c)
\end{tabular}
\end{center}
\caption{GPE with angular momentum: $\omega = 0.75$, $\beta = 100$. (a) Computed ground-state density $|u_N|^2$; (b) Comparison of the error with respect to the square root of the
number of degrees of freedom; (c) Final $hp$-mesh.}
\end{figure}

Numerical experiments for each of the values of $\omega$ and $\beta$ given in Table~\ref{table:energies_for_positive_omega} are presented in Figures~\ref{fig:omega_0.5_beta_10}--\ref{fig:omega_0.9_beta_1000}. As before, we show the computed ground-state, the comparison between the error in the approximated ground-state energy when employing both $hp$-- and $h$--refinement, and the final $hp$--mesh. We note that as $\omega$ is increased, the number of vortices present in the ground-state increases; for the case when $\omega = 0.9$ and $\beta = 1000$, the ground-state features over 50 vortices arranged in a regular triangular lattice referred to as the Abrikosov lattice. In all cases we again observe that the proposed $hp$--refinement algorithm leads to exponential convergence of the error in the approximated ground-state energy; moreover, the efficiency of employing $hp$--refinement when compared to mesh subdivision alone is clearly highlighted. As the number of vortices increases with each of the test cases considered here, we observe that the accuracy attained for the computed ground-state energy decreases. Moreover, we also observe that more refinement of the computational mesh is necessary, before subsequent $p$--refinement is employed to yield exponential convergence of the error in the ground-state energy, as $\omega$ is increased. In the simplest setting, i.e., when $\omega=0.5$ and $\beta=10$ we observe that the proposed $hp$--refinement algorithm simply defaults to polynomial enrichment only.

\begin{figure}[t!]
\begin{center}
\begin{tabular}{cc}
\includegraphics[width=0.45\textwidth]{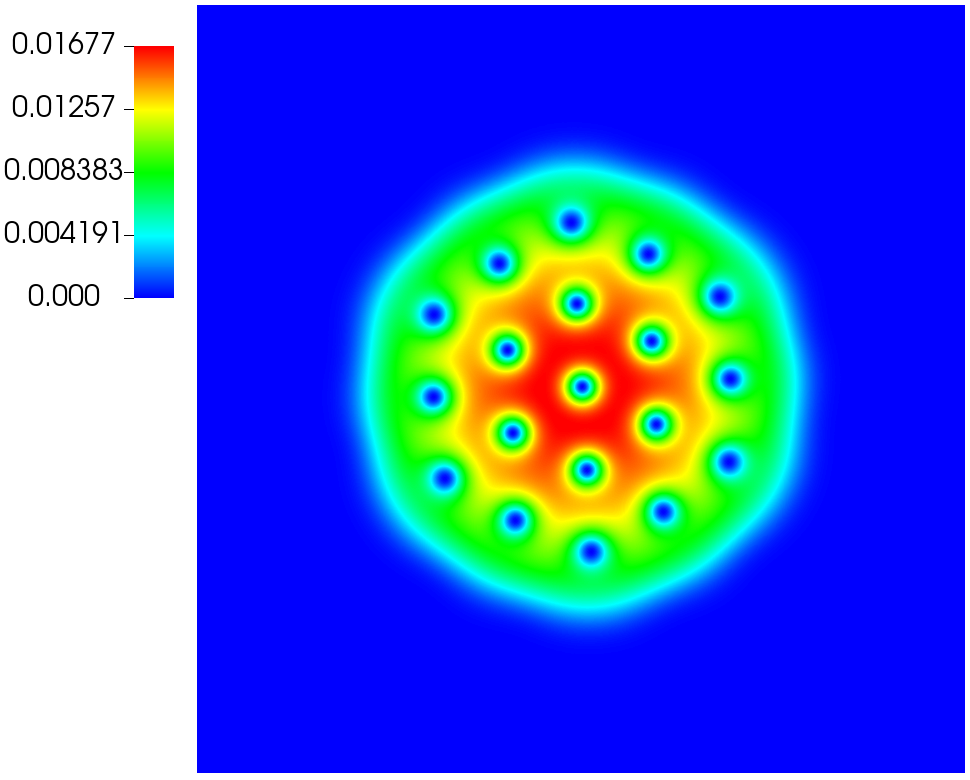} &
\includegraphics[width=0.45\textwidth]{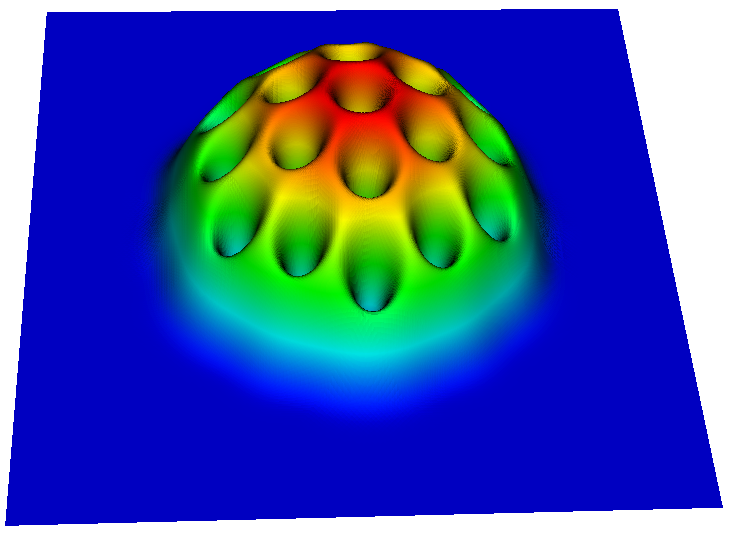} \\
\multicolumn{2}{c}{(a)} \\
& \\
\includegraphics[scale=0.4]{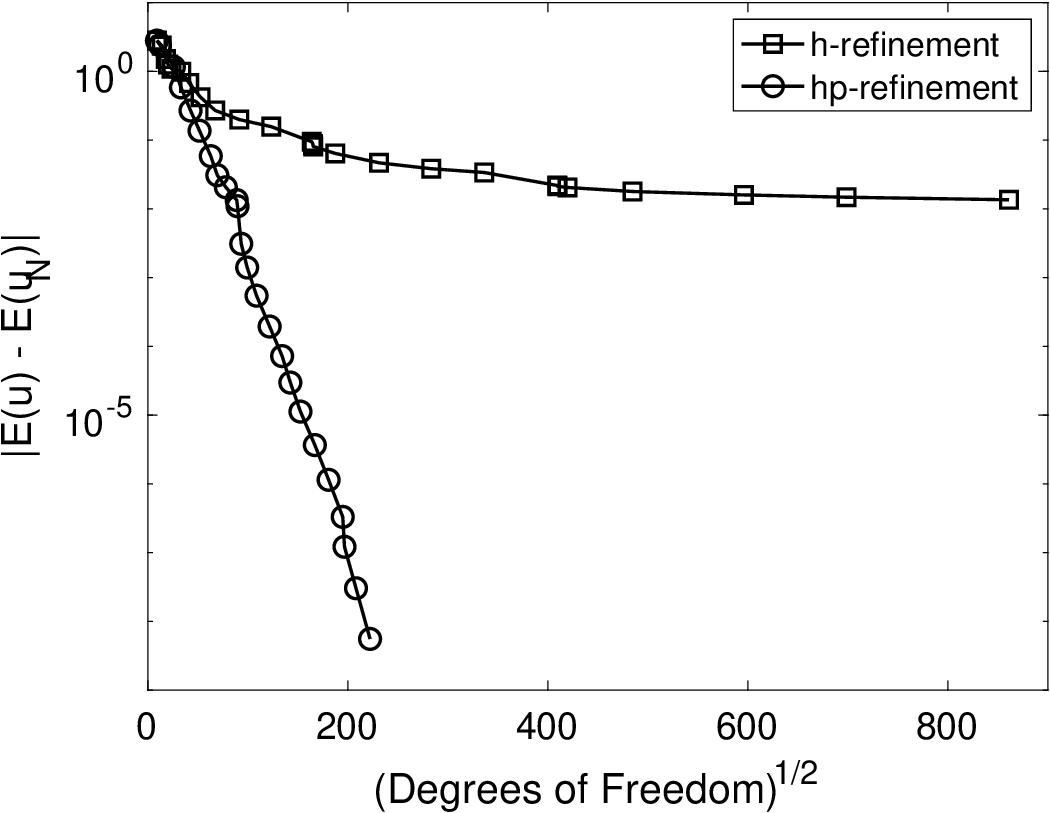} &
\includegraphics[scale=0.43]{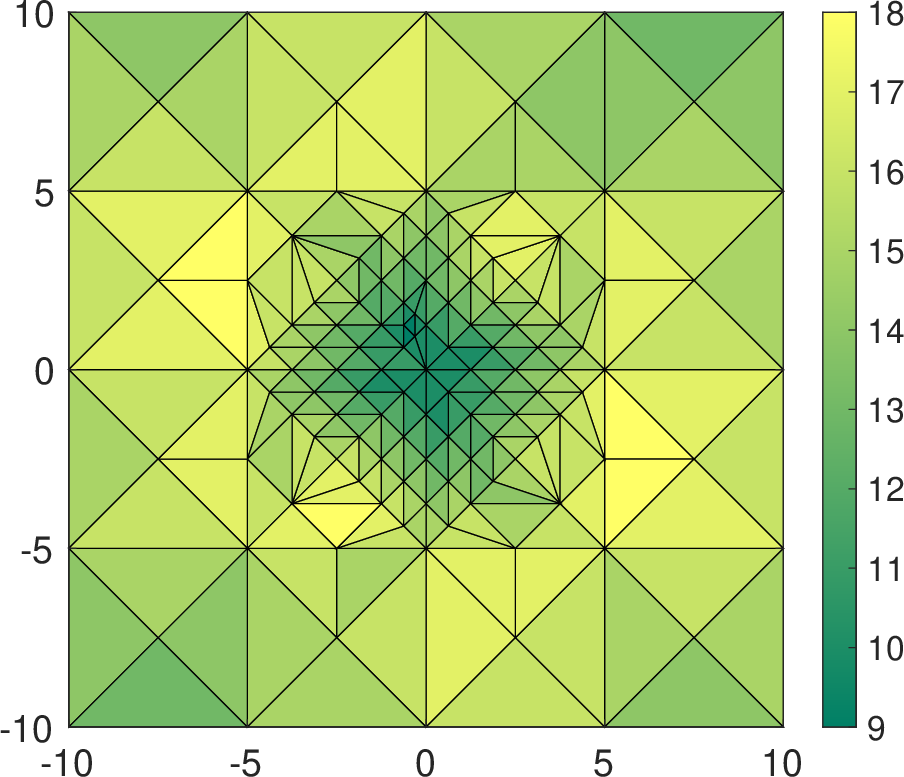} \\
(b) & (c)
\end{tabular}
\end{center}
\caption{GPE with angular momentum: $\omega = 0.8$, $\beta = 500$. (a) Computed ground-state density $|u_N|^2$; (b) Comparison of the error with respect to the square root of the
number of degrees of freedom; (c) Final $hp$-mesh.}
\end{figure}

\begin{figure}[t!]
\begin{center}
\begin{tabular}{cc}
\includegraphics[width=0.45\textwidth]{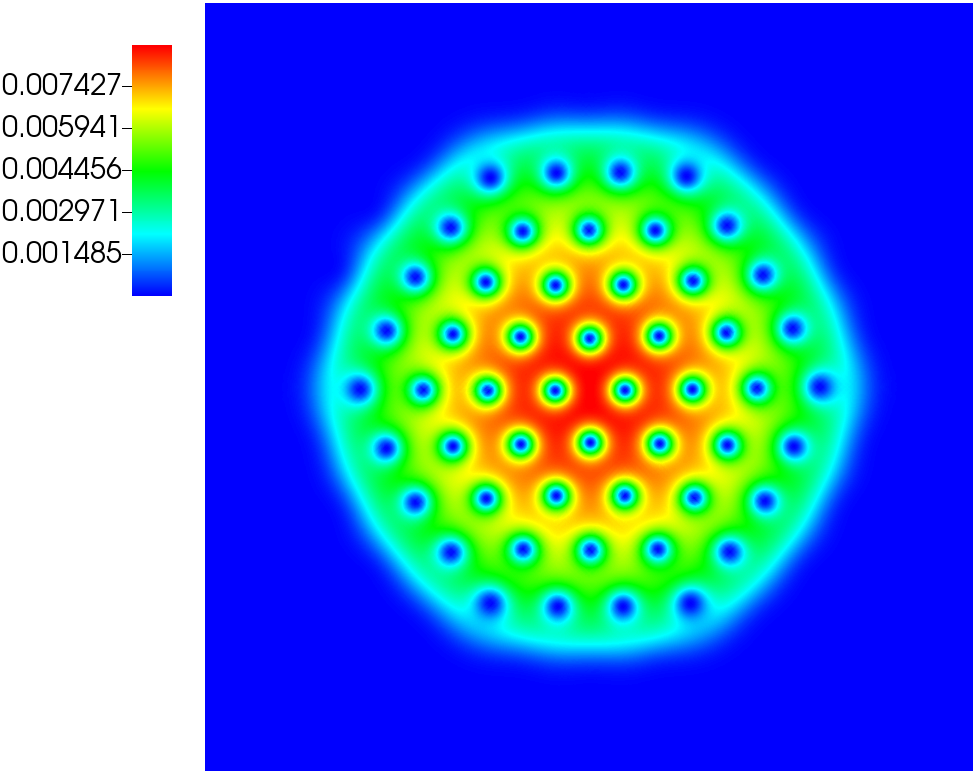} &
\includegraphics[width=0.5\textwidth]{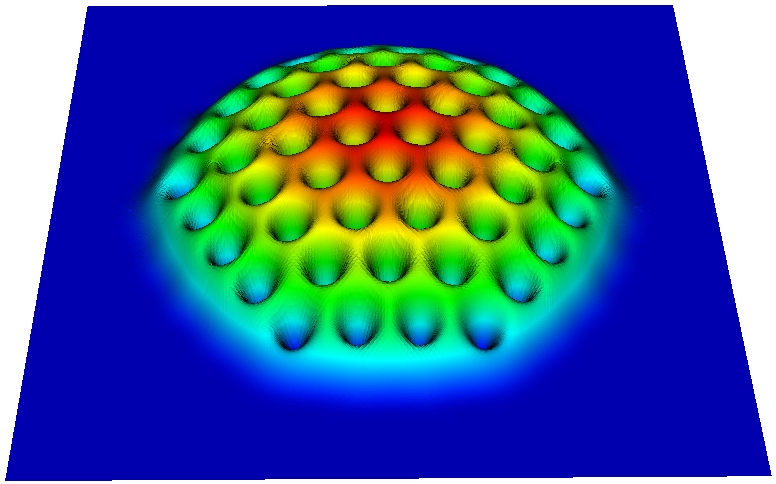} \\
\multicolumn{2}{c}{(a)} \\
& \\
\includegraphics[scale=0.4]{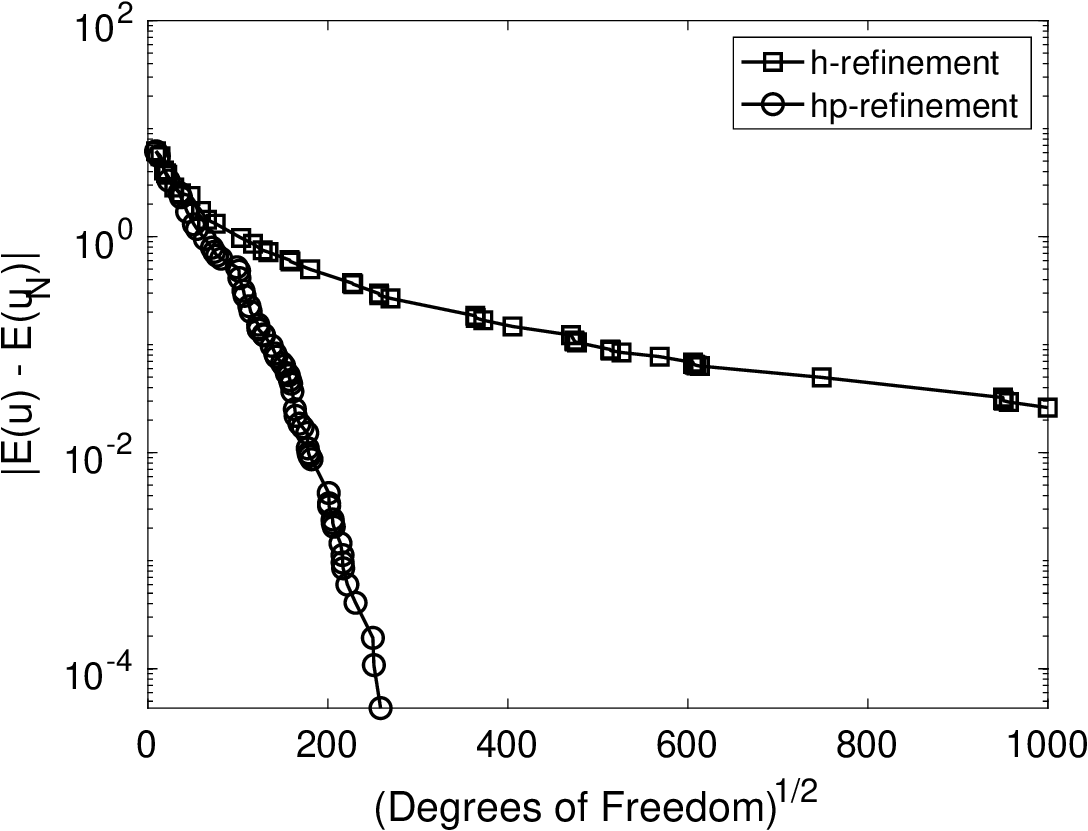} &
\includegraphics[scale=0.43]{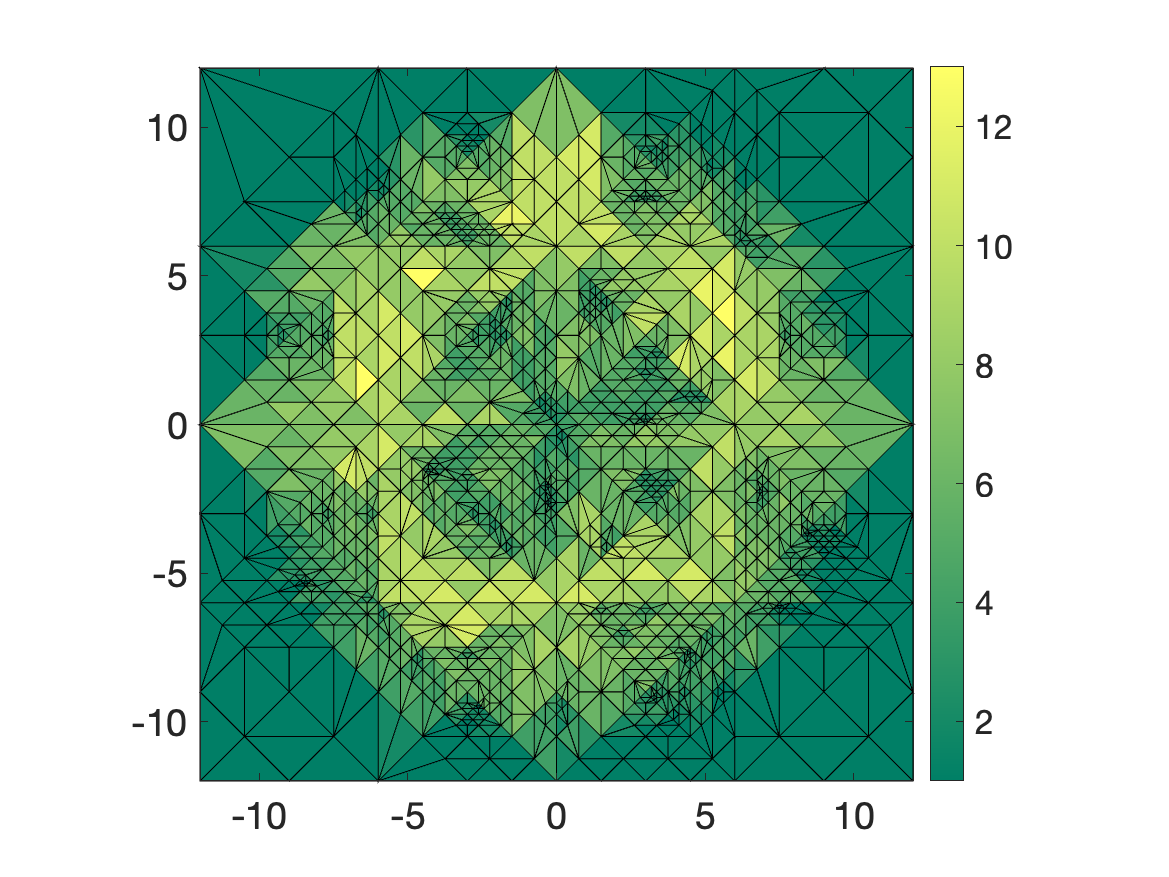} \\
(b) & (c)
\end{tabular}
\end{center}
\caption{GPE with angular momentum: $\omega = 0.9$, $\beta = 1000$. (a) Computed ground-state density $|u_N|^2$; (b) Comparison of the error with respect to the square root of the
number of degrees of freedom; (c) Final $hp$-mesh.}
\label{fig:omega_0.9_beta_1000}
\end{figure}

\section{Conclusions} \label{sec:conclusions}
In this paper, we have considered the stationary non-linear Gross--Pitaevskii eigenvalue problem in the \emph{presence of a rotating magnetic field}, which yields an intricate topology manifested by the presence of quantized vortices in the ground-state; the \emph{number and location of these vortices is unknown a priori}. In particular, we have extended the energy-based adaptive finite element strategy presented in ~\cite{HeidStammWihler:21} in two key directions: firstly, we have considered the case when a rotating magnetic field is present in the underlying energy functional, and secondly, we have generalised the adaptivity framework to allow for $hp$-refinements of the underlying discrete finite element approximation space. 
Towards these goals, we have also presented a result on the convergence of a discrete iteration scheme  (Theorem~\ref{thm:discreteconvergence}) which is based on a continuous projected gradient flow that respects the energy-based topology of the underlying problem at hand. The key feature of the proposed energy-adaptive $hp$-mesh refinement strategy is that it is based solely on potential energy-decay which can be computed locally and does not require sophisticated localisation schemes as is standard for residual-based a posteriori error estimators. 
The resulting method thereby intertwines the $hp$-mesh refinement iterative procedure with the iterative non-linear solver for the underlying problem within a unified framework in the sense that the numerical approximation to the non-linear problem is computed in tandem while the $hp$-mesh is adaptively refined.
The performance of this proposed strategy is demonstrated on a  diverse set of benchmark test problems that highlight the exponential convergence of the error in the computed ground-state energy with respect to the number of degrees of freedom in the resulting $hp$-finite element space.

\section*{Acknowledgments}
T.~P.~W. acknowledges the financial support of the Swiss National Science Foundation (SNF), Grant No.~$200021\underline{\phantom{-}}212868$.
B.~S. acknowledges funding by the Deutsche Forschungsgemeinschaft (DFG, German Research Foundation), Project No.~442047500, through the Collaborative Research Center “Sparsity and Singular Structures” (SFB 1481).
Finally, we thank Eric Canc\`{e}s for encouraging discussions and motivating us to extend the framework to the case with rotating magnetic field. 

\bibliographystyle{amsplain}
\bibliography{references,GPrefs}

\end{document}